\documentclass[10pt,oneside,english,final]{amsart}
\usepackage[T1]{fontenc}
\usepackage[latin9]{inputenc}
\usepackage{geometry}
\usepackage{amsthm}
\usepackage{amstext}
\usepackage{amssymb}
\usepackage{esint}
\usepackage[numbers]{natbib}

\makeatletter
\numberwithin{equation}{section}
\numberwithin{figure}{section}
\theoremstyle{plain}
\newtheorem{thm}{\protect\theoremname}
  \theoremstyle{plain}
  \newtheorem{fact}[thm]{\protect\factname}
  \theoremstyle{remark}
  \newtheorem{rem}[thm]{\protect\remarkname}
  \theoremstyle{definition}
  \newtheorem{defn}[thm]{\protect\definitionname}
  \theoremstyle{plain}
  \newtheorem{cor}[thm]{\protect\corollaryname}
  \theoremstyle{plain}
  \newtheorem{lem}[thm]{\protect\lemmaname}

\usepackage{color}
\usepackage[unicode=true,
 bookmarks=false,
 breaklinks=false,pdfborder={0 0 1},backref=false,colorlinks=false]
 {hyperref}
\usepackage[notref,notcite]{showkeys} 

\makeatother

\usepackage{babel}
  \providecommand{\corollaryname}{Corollary}
  \providecommand{\definitionname}{Definition}
  \providecommand{\factname}{Fact}
  \providecommand{\lemmaname}{Lemma}
  \providecommand{\remarkname}{Remark}
\providecommand{\theoremname}{Theorem}

\begin{document}
\global\long\def\intr{\int_{R}}
 \global\long\def\sbr#1{\left[ #1\right] }
\global\long\def\cbr#1{\left\{  #1\right\}  }
\global\long\def\rbr#1{\left(#1\right)}
\global\long\def\ev#1{\mathbb{E}{#1}}
\global\long\def\R{\mathbb{R}}
\global\long\def\norm#1#2#3{\Vert#1\Vert_{#2}^{#3}}
\global\long\def\pr#1{\mathbb{P}\rbr{#1}}
\global\long\def\cleq{\lesssim}
\global\long\def\ceq{\eqsim}
\global\long\def\conv{\rightarrow}
\global\long\def\Var#1{\text{Var}(#1)}
\global\long\def\TDD#1{{\color{red}To\, Do(#1)}}
\global\long\def\dd#1{\textnormal{d}#1}
\global\long\def\inti{\int_{0}^{\infty}}
\global\long\def\crr{\mathcal{C}([0,\infty),\R)}
\global\long\def\sb#1{\langle#1\rangle}
\global\long\def\pm#1{d_{P}\rbr{#1}}
\global\long\def\crt{\mathcal{C}([0,T],\R)}
\global\long\def\nuu{\nu_{n;\lambda}}
\global\long\def\ZZ{Z_{\Lambda_{n}}}
\global\long\def\PP{\mathbb{P}_{\Lambda_{n}}}
\global\long\def\EE{\mathbb{E}_{\Lambda_{n}}}
\global\long\def\LL{\Lambda_{n}}
\global\long\def\AA{\mathcal{A}}
\global\long\def\evx{\mathbb{E}_{x}}
\global\long\def\pin#1{1_{\cbr{#1\in\mathcal{A}}}}
\global\long\def\Zd{\mathbb{Z}^{d}}
\global\long\def\TT{T_{n}}
\global\long\def\ZZa{Z_{n;\epsilon}}
\global\long\def\emax{\bar{\epsilon}}
\global\long\def\emin{\underbar{\ensuremath{\epsilon}}}
\global\long\def\estar{\epsilon^{*}}
\global\long\def\ee{\mathbf{e}}
\global\long\def\ddp#1#2{\langle#1,#2\rangle}
\global\long\def\intc#1{\int_{0}^{#1}}
\global\long\def\T#1{\mathcal{P}_{#1}}
\global\long\def\ii{\mathbf{i}}
\global\long\def\star#1{\left.#1^{*}\right.}
\global\long\def\pspace{\mathcal{C}}
\global\long\def\eq{\varphi}
\global\long\def\grad{\text{grad}}
\global\long\def\var{\text{Var}}
\global\long\def\fMeas{\mathcal{M}_{F}(\mathbb{R}^{d})}

\title{Spatial CLT for the supercritical Ornstein-Uhlenbeck superprocess}

\author{Piotr Mi\l{}o\'{s}\\
Faculty of Mathematics, Informatics and Mechanics, University of Warsaw}

\thanks{This research was partially supported by a Polish Ministry of Science
grant N N201 397537 and by the British Council Young Scientists Programme.}

\date{02.01.2011}
\begin{abstract}
In this paper we consider a superprocess being a measure-valued diffusion
corresponding to the equation $u_{t}=Lu+\alpha u-\beta u^{2}$, where
$L$ is the infinitesimal operator of the \emph{Ornstein-Uhlenbeck
process} and $\beta>0,\:\alpha>0$. The latter condition implies that
the process is \emph{supercritical, }i.e. its total mass grows exponentially.
This system is known to fulfill a law of large numbers. In the paper
we prove the corresponding \emph{central limit theorem}. The limit
and the CLT normalization fall into three qualitatively different
classes. In what we call the small growth rate case the situation
resembles the classical CLT. The weak limit is Gaussian and the normalization
is the square root of the size of the system. In the critical case
the limit is still Gaussian, however the normalization requires an
additional term. Finally, when the growth rate is large the situation
is completely different. The limit is no longer Gaussian, the normalization
is substantially larger than the classical one and the convergence
holds in probability. These different regimes arise as a result of
{}``competition'' between spatial smoothing due to the particles'
movement and the system's growth which is local.

We prove also that the spatial fluctuations are asymptotically independent
of the fluctuations of the total mass of the process (which is a continuous
State Branching Processes). 
\end{abstract}
\maketitle

\section{Introduction}

\subsection{Model}

Let $\cbr{\T t}_{t\geq0,}$ be the semigroup of the Ornstein-Uhlenbeck
process in $\R^{d}$ i.e. the one with the infinitesimal operator
\begin{equation}
L:=\frac{1}{2}\sigma^{2}\Delta-\mu x\circ\nabla,\quad\sigma>0,\:\mu>0,\label{eq:infinitesimalOP}
\end{equation}
where $\circ$ denotes the standard scalar product. Abusing the notation
we will denote both the invariant distribution (of $L$) and its density
with the same symbol, namely
\begin{equation}
\eq(x)=\rbr{\frac{\mu}{\pi\sigma^{2}}}^{d/2}\exp\rbr{-\frac{\mu}{\sigma^{2}}\norm x{}2}.\label{eq:equilibrium}
\end{equation}
In this paper we will study the behavior of the superprocess $\cbr{X_{t}}_{t\geq0}$
with semigroup $\T{}$ and the branching mechanism $\psi$ given by
\begin{equation}
\psi(\lambda)=-\alpha\lambda+\beta\lambda^{2},\quad\alpha\in\R,\beta>0.\label{eq:branchingMechnism}
\end{equation}
Its behavior is described by a Markovian evolution; that is, for each
$\nu\in\fMeas$ being the set of finite, compactly supported measures,
by $\mathbb{P}_{\nu}$ we denote the the law of $X$ with initial
condition $\nu$. Let $f\in bp(\mathbb{R}^{d})$ (bounded, positive
and measurable functions on $\R^{d}$) . The following equation defines
the evolution (see \citep{Dynkin:1993sw}, \citep{Dynkin:2002ry}
for more details). 
\begin{equation}
-\log\mathbb{E}_{\nu}(e^{-\langle f,X_{t}\rangle})=\int_{\mathbb{R}^{d}}u_{f}(x,t)\nu({\rm d}x),\quad t\geq0,\label{eq:logLaplace}
\end{equation}
 where $u_{f}(x,t)$ is the unique non-negative solution to the integral
equation 
\begin{equation}
u_{f}(x,t)=\T tf(x)-\int_{0}^{t}\T s[\psi(u_{f}(\cdot,t-s))](x)\dd s.\label{eq:integralEq}
\end{equation}
In the paper we are interested only in the \emph{supercritical} case,
i.e. the case when the total mass of the superprocess grows (exponentially).
This is ensured by the condition

\[
\alpha>0,
\]
which we assume in the whole paper. The above definition is rather
abstract but superprocess have a natural interpretation as the short
life time and high density diffusion limit of a branching particle
systems (see e.g. Introduction of \citep{Englnder:2006wm} and Remark
\ref{rem:BPSapproximation}). In our case the branching particle counterpart
is the Ornstein-Uhlenbeck branching system considered in \citep{Adamczak:2011kx,Adamczak:2011fk}.
We will comment on this connection later on in Remark \ref{rem:BPSapproximation}.

\subsection{Results}

The expectation of the total mass of the system grows exponentially
fast at a rate $\alpha$. Furthermore, the system fulfills the law
of large numbers \citep[Theorem 1]{Englnder:2006wm}, i.e. for a bounded
continuous function $f$ we have 
\[
\lim_{t\conv+\infty}e^{-\alpha t}\ddp{X_{t}}f=\ddp{\varphi}fV_{\infty},\quad\text{in probability},
\]
where $V_{\infty}$ is a random variable, to be defined later (let
us note that the results of \citep{Englnder:2006wm} holds for a larger
class of branching diffusions). The goal of our paper is to prove
the corresponding \emph{central limit theorem}. The second order behavior
depends qualitatively on the sign of $\alpha-2\mu$. Roughly speaking
this condition reflects the interplay of two antagonistic forces,
the growth which is local and makes the system more coarse and the
smoothing introduced by the spatial evolution corresponding to the
OU-process. We will now describe this behavior in more detail. The
main object of our interest are the spatial fluctuations given by:
\[
F_{t}^{-1}\rbr{\ddp{X_{t}}f-|X_{t}|\ddp{\eq}f},
\]
where $F_{t}$ is some norming, not necessarily deterministic, and
$|X_{t}|:=\ddp{X_{t}}1$ is the total mass of the system. We will
describe the situation on the set where the process is not extinguished
$Ext^{c}$ (to be defined later). The results split into three classes
\begin{description}
\item [{Small~growth~rate~$\alpha<2\mu$}] Our main result is contained
in Theorem \ref{thm:smallBranching}. In this the case ``the movement
part prevails'' and the result resembles the standard CLT. The normalization
is given by $F_{t}=|X_{t}|^{1/2}$ (which is of order $e^{-(\alpha/2)t}$).
Moreover, we obtain the limit which is Gaussian (though its variance
is given by a complicated formula) and does not depend on the starting
configuration. 
\item [{Critical~growth~rate~$\alpha=2\mu$}] Our main result is contained
in Theorem \ref{thm:criticalBranching}. In this case ``the growth
prevails''. The behavior of the fluctuations slightly diverge from
the classical setting. The normalization is bigger: $F_{t}=t^{1/2}|X_{t}|^{1/2}$.
The limit still does not depend on the starting condition and is Gaussian
but its variance depends on the derivatives of $f$. To explain this
we notice that the growth is so fast that the fluctuations are not
smoothed by the motion and become essentially local. In consequence
they give rise to a {}``spatial white noise'' and larger normalization
is required.
\item [{Large~growth~rate~$\alpha>2\mu$}] Our main result is contained
in Theorem \ref{thm:fastBranching}. In this case not only does the
growth ``prevail'' but also ``the motion fails to make any smoothing''.
The normalization is even bigger: $F_{t}=e^{(\alpha-\mu)t}$ and we
have $\alpha-\mu>\alpha/2$. The limit is no longer Gaussian, it is
given by $\ddp f{\grad\varphi}\circ\tilde{H}_{\infty}$ (where $\tilde{H}_{\infty}$
is the limit of a certain martingale). What is perhaps surprising,
the limit holds in probability. The first term, $\ddp f{\grad\varphi}$,
means that like in the critical situation the growth is fast enough
to produce some sort of a white noise. Even more, it is so fast that
the limit depends on the starting condition and in fact, up to some
extent, the system ``remembers its whole evolution'', which is encoded
in $\tilde{H}_{\infty}$. 
\end{description}
In either case we prove also that the spatial fluctuations become
independent of fluctuations of the total mass as time increases.

\subsection{Comments}

This paper is a superprocess counterpart of \citep{Adamczak:2011kx}.
Namely, $X$ can be defined as a limit of the Ornstein-Uhlenbeck particle
systems considered in \citep{Adamczak:2011kx} (see Remark \ref{rem:BPSapproximation}).
It turns out that qualitatively the first and second order behavior
are very similar (the reader aware of \citep{Adamczak:2011kx} easily
notices that the list above is very similar to the one in \citep{Adamczak:2011kx}).
Related problems for branching particle systems were considered also
in \citep{Adamczak:2011fk,Bansaye:2009cl} however we are not familiar
of any results in this direction concerning superprocesses. For general
information about superprocesses we refer to books \citep{Dynkin:1994aa,Etheridge:2000fe,Gall:1999pi}
and also to the references given in Introduction of \citep{Adamczak:2011kx}.

Let us now comment on the methodology. Our proofs hinge on the backbone
representation obtained recently in \citep{Berestycki:2011kx} which
allows us to reuse many concepts developed in \citep{Adamczak:2011kx}.
This together with analytical estimation of the behavior of $\T{}$
is enough to prove our results. A more detailed description of the
proof startegy is in Remark \ref{rem:proofStrategy}.

The article is organized as follows. The next section presents notation
and basic fact required further. Section \ref{sec:Results} is devoted
to presentation of results. Proofs are deferred to Sections \ref{sec:Proof-Preliminaires}-\ref{sec:criticalBranching}
and the Appendix.

\section{Preliminaries and notation\label{sec:Preliminareis}}

Let us first recall the notions which appeared in Introduction. $\T{}$
is the semigroup corresponding to \eqref{eq:infinitesimalOP}. $\mathcal{M}_{F}$
is the space of finite, compactly supported measures and $bp(\R^{d})$
is the space of bounded, positive and measurable functions on $\R^{d}$. 

We use $\ddp f{\nu}:=\int_{\R^{d}}f(x)\nu(\dd x)$. We denote the
total mass of the measure $\nu$ by $|\nu|:=\ddp 1{\nu}$. 

By $x\cleq y$ we will denote the fact that $x\leq Cy$ for a certain
constant $C>0$ (which exact value is not relevant to following calculations). 

We will use 
\begin{multline*}
\pspace{}=\pspace{}(\R^{d}):=\left\{ f:\R^{d}\mapsto\R:f\,\text{ is continuous and }\right.\\
\left.\text{there exists }n\text{ such that }|f(x)|/\norm x{}n\conv0\text{ as }\norm x{}{}\conv+\infty\right\} ,
\end{multline*}
that is the space of continuous functions which grow at most polynomially.
To shorten the notation we introduce $\cbr{\T t^{\alpha}}_{t\geq0}$
by
\[
\T t^{a}f(x):=e^{at}\T tf(x).
\]
We may rewrite equation on $u_{f}$ as 
\begin{equation}
u_{f}(x,t)=\T t^{\alpha}f(x)-\beta\int_{0}^{t}\T s^{\alpha}[u_{f}(\cdot,t-s)^{2}](x)\dd s.\label{eq:integralEquation}
\end{equation}
By $Ext$ we denote the event that the process is extinguished, i.e.
\begin{equation}
Ext:=\cbr{\lim_{t\conv+\infty}|X_{t}|=0}.\label{eq:extenquished}
\end{equation}
We denote also process $\cbr{V_{t}}_{t\geq0}$ by 
\begin{equation}
V_{t}:=e^{-\alpha t}|X_{t}|,\label{eq:martingaleDef}
\end{equation}
The gather the basic facts of process $V$ which will be useful in
the formulation of the main results
\begin{fact}
\label{fac:totalMassMartingale}Let $\cbr{X_{t}}_{t\geq0}$ be the
OU-superprocess starting from $\nu\in\fMeas$ and let $\cbr{V_{t}}_{t\geq0}$
be defined according to \eqref{eq:martingaleDef}. Then $V$ with
it natural filtration is a (positive) martingale. It converges 
\begin{equation}
V_{\infty}:=\lim_{t\conv+\infty}V_{t},\quad\text{a.s. and in }L^{2}.\label{eq:Vlimit}
\end{equation}
Moreover, $\cbr{V_{\infty}=0}=Ext$, $\mathbb{P}_{\nu}(Ext)=\exp\cbr{-|\nu|\frac{\alpha}{\beta}}$
and the law of $V_{\infty}$ can be described by 
\begin{equation}
V_{\infty}=^{d}\sum_{i=1}^{N}E_{i},\label{eq:LimitRep}
\end{equation}
where $N$ is a Poisson random variable with parameter $|\nu|\frac{\alpha}{\beta}$
and $E_{1},E_{2},\ldots$ is an i.i.d sequence of exponential random
variables with parameter $\alpha/\beta$, which are independent of
$N$. We also have
\begin{equation}
\sigma_{V}^{2}:=Var(V_{\infty}):=\frac{2\beta}{\alpha}|\nu|.\label{eq:varianceMaringaleLimit}
\end{equation}

\end{fact}
The proof of the fact is delegated to Appendix.
\begin{rem}
Although representation \eqref{eq:LimitRep} may seem a little peculiar
at the moment it will appear naturally once we introduce the backbone
decomposition in Section \ref{sub:Backbone-construction}.
\end{rem}

\section{Results\label{sec:Results}}

In this section we present the results of our paper. We split the
section into three parts corresponding to each case listed in Introduction

\subsection{Slow growth  $\alpha<2\mu$.\protect \\
}

Let us denote $\tilde{f}(x):=f(x)-\ddp{\varphi}f$ and 
\begin{equation}
\sigma_{f}^{2}:=\frac{\beta}{\alpha}\int_{0}^{\infty}e^{-\alpha s}\ddp{\varphi}{2\beta\rbr{\T s^{\alpha}\tilde{f}(\cdot)}^{2}-2\beta\rbr{\T s^{-\alpha}\tilde{f}(\cdot)}^{2}+4\alpha\beta u(\cdot,s)}\dd s,\label{eq:sigmaDefinition}
\end{equation}
where $u(x,s)=\intc s\T{s-u}^{-\alpha}\sbr{\rbr{\T u^{-\alpha}\tilde{f}(\cdot)}^{2}}(x)\dd u.$

The main result of this section is
\begin{thm}
\label{thm:smallBranching}Let $\cbr{X_{t}}_{t\geq0}$ be the OU-superprocess
starting from $\nu\in\fMeas$. Let us assume $\alpha<2\mu$ and $f\in\pspace(\R^{d})$.
Then $\sigma_{f}<+\infty$ and conditionally on set $Ext^{c}$ there
is the convergence
\[
\rbr{e^{-\alpha t}|X_{t}|,\frac{|X_{t}|-e^{\alpha t}V_{\infty}}{\sqrt{|X_{t}|}},\frac{\ddp{X_{t}}f-|X_{t}|\ddp f{\varphi}}{\sqrt{|X_{t}|}}}\rightarrow^{d}(\tilde{V}_{\infty},G_{1},G_{2}),\quad\text{as }t\conv+\infty,
\]
where $G_{1}\sim\mathcal{N}(0,\frac{2\beta}{\alpha}),G_{2}\sim\mathcal{N}(0,\sigma_{f}^{2})$
and $\tilde{V}_{\infty}$ is $V_{\infty}$ conditioned on $Ext^{c}$.
Moreover, the variables $\tilde{V}_{\infty},G_{1},G_{2}$ are independent
.
\end{thm}
The proofs corresponding to this case are delegated to Section \ref{sec:smallBranching}.

\subsection{Critical growth $\alpha=2\mu$.\protect \\
}

We denote 
\begin{equation}
\sigma_{f}^{2}:=2\frac{\beta^{2}}{\alpha}\int_{\R^{d}}\rbr{x\circ\ddp{\grad f}{\eq}}^{2}\varphi(x)\dd x,\label{eq:criticalSigma}
\end{equation}
where $\circ$ is the standard scalar product. The main result of
this section is 
\begin{thm}
\label{thm:criticalBranching}Let $\cbr{X_{t}}_{t\geq0}$ be the OU-superprocess
starting from $\nu\in\fMeas$. Let us assume $\alpha=2\mu$ and $f\in\pspace(\R^{d})$.
Then $\sigma_{f}^{2}<+\infty$ and conditionally on set $Ext^{c}$
there is the convergence 
\[
\rbr{e^{-\alpha t}|X_{t}|,\frac{|X_{t}|-e^{t\alpha}V_{\infty}}{\sqrt{|X_{t}|}},\frac{\ddp{X_{t}}f-|X_{t}|\ddp f{\eq}}{t^{1/2}\sqrt{|X_{t}|}}}\rightarrow^{d}(\tilde{V}_{\infty},G_{1},G_{2}),\quad\text{as }t\conv+\infty,
\]
where $G_{1}\sim\mathcal{N}(0,\frac{2\beta}{\alpha}),G_{2}\sim\mathcal{N}(0,\sigma_{f}^{2})$
and $\tilde{V}_{\infty}$ is $V_{\infty}$ conditioned on $Ext^{c}$.
Moreover, the variables $\tilde{V}_{\infty},G_{1},G_{2}$ are independent.
\end{thm}
The proofs corresponding to this case are delegated to Section \ref{sec:criticalBranching}.

\subsection{Fast growth $\alpha>2\mu$.\protect \\
}

Let us define a process $\cbr{H_{t}}_{t\geq0}$ by 
\begin{equation}
H_{t}:=e^{-(\alpha-\mu)t}\int_{\R^{d}}x\: X_{t}(\dd x).\label{eq:secondMartingale}
\end{equation}

\begin{fact}
\label{fac:Hmartingale}The process $H$ is a martingale and under
assumption $\alpha>2\mu$ it is $L^{2}$-bounded.

From the fact it follows that in the setting of this section the following
limit (both $a.s.$ and $L^{2}$) is well-defined and $a.s.$ finite
\[
H_{\infty}:=\lim_{t\conv+\infty}H_{t}.
\]

\end{fact}
Let us note that $H_{\infty}$ depends on the starting condition $\nu\in\mathcal{M}_{F}(\R^{d})$,
describing this dependence would be notationally cumbersome at this
point, hence is deferred to Theorem \ref{fact:Hdependence}.

The main result of this section is 
\begin{thm}
\label{thm:fastBranching}Let $\cbr{X_{t}}_{t\geq0}$ be the OU-superprocess
starting from $\nu\in\fMeas$. Let us assume $\alpha>2\mu$ and $f\in\pspace(\R^{d})$.
Then conditionally on the set of non-extinction $Ext^{c}$ there is
the convergence 
\begin{equation}
\rbr{e^{-\alpha t}|X_{t}|,\frac{|X_{t}|-e^{t\alpha}V_{\infty}}{\sqrt{|X_{t}|}},\frac{\ddp{X_{t}}f-|X_{t}|\ddp f{\eq}}{\exp\rbr{(\alpha-\mu)t}}}\rightarrow^{d}(\tilde{V}_{\infty},G,\ddp{\grad f}{\eq}\circ\tilde{H}_{\infty}),\label{eq:factConv1}
\end{equation}
where $G\sim\mathcal{N}(0,\frac{2\beta}{\alpha})$, variables $\tilde{H}_{\infty},\tilde{V}_{\infty}$
are respectively $H_{\infty},V_{\infty}$ conditioned on $Ext^{c}$
and $(\tilde{V}_{\infty},\tilde{H}_{\infty}),G$ are independent.
Moreover 
\begin{equation}
\rbr{e^{-\alpha t}|X_{t}|,\frac{\ddp{X_{t}}f-|X_{t}|\ddp f{\eq}}{\exp\rbr{(\alpha-\mu)t}}}\rightarrow(V_{\infty},\ddp{\grad f}{\eq}\circ H_{\infty}),\quad\text{in probability.}\label{eq:fastConv2}
\end{equation}

\end{thm}
The proofs corresponding to this case are delegated to Section \ref{sec:fastBranching}.

\subsection{Discussion and remarks.\protect \\
}

As a corollary to the above statements we obtain a weak law of large
numbers for a slightly larger class of the test functions, compared
to \citep[Theorem 1]{Englnder:2006wm} . 
\begin{thm}
Let $\cbr{X_{t}}_{t\geq0}$ be the OU-superprocess starting from $\nu\in\fMeas$
and $f\in\pspace(\R^{d})$ then 
\[
\lim_{t\conv+\infty}e^{-\alpha t}\ddp{X_{t}}f=\ddp{\varphi}tV_{\infty},\quad\text{in probability.}
\]
\end{thm}
\begin{rem}
\label{rem:BPSapproximation}As mentioned in the Introduction the
results are closely related to the ones in \citep{Adamczak:2011kx}.
This follows naturally by the fact that $X$ can be defined in terms
of the OU branching systems considered therein. This construction
can be described as follows. In the $n$-th approximation each particles
carries mass $1/n$ and lives for an exponential time with parameter
$1/n$. During this time it executes a random movement according to
the Ornstein-Uhlenbeck process with the infinitesimal operator $L$.
When it dies the particle is replaced by a random number of offspring.
The mean of this number is $1+\alpha/n$, while the variance is $2\beta$. 

The particle view point gives more intuition. Having this in mind
it is easier to understand the discussion in the Introduction, moreover
some further heuristics are given in \citep[Remark 3.3, Remark 3.7, Remark 3.11]{Adamczak:2011kx}. 
\end{rem}

\begin{rem}
The law of $H_{\infty}$ is an unresolved problem. It can be proved
however that it is not Gaussian. Using Fact \ref{fact:Hdependence}
we will see that $H_{\infty}$ is closely related to the corresponding
limit for the Ornstein-Uhlenbeck branching process. For further discussion
we refer the reader to \citep[Remark 3.12]{Adamczak:2011kx}.
\end{rem}

\begin{rem}
We suspect that the convergence in \eqref{eq:fastConv2} is in fact
almost sure.
\end{rem}

\begin{rem}
In our paper, for the sake of simplicity, we choose to work with the
branching mechanism \eqref{eq:branchingMechnism}. Using our methods
it should be straightforward to prove the results for any branching
mechanism which admits the fourth moment (i.e. $\psi^{(4)}(0)<+\infty$). 

We conjecture also that the results are valid for any branching mechanism
with the second moment. An interesting question would be to go beyond
this assumption. It is natural to expect different normalization and
convergence to some stable random variable.
\end{rem}

\begin{rem}
The other remarks of \citep{Adamczak:2011kx} are also relevant to
our case. For the sake of brevity we only mention that the most important
extensions of the present paper will be to study superprocesses with
general diffusion, instead for the Ornstein-Uhlenbeck process and
the case of superprocesses with non-homogenous branching rates. This
will be by no means a trivial task, a short explanation of the forthcoming
difficulties is given in \citep[Remark 3.16]{Adamczak:2011kx}.
\end{rem}

\section{Proof Preliminaries\label{sec:Proof-Preliminaires}}

In this section we gather all auxiliary facts used in the proofs of
results presented in Section \ref{sec:Results}. The proof itself
are contained separately in Sections \ref{sec:smallBranching}-\ref{sec:criticalBranching}.

\subsection{Backbone construction\label{sub:Backbone-construction}}

Following recent developments, e.g. \citep{Berestycki:2011kx}, we
present a so-called backbone construction of a \emph{supercritical}
superprocess. The main idea is to define a \emph{backbone}, being
a \emph{supercritical branching particle system}, which is dressed
with \emph{subcritical} superprocesses. This allows, up to some extent,
to treat a superprocess as a discrete object. Such property makes
things easier. On the conceptual level the proofs presented in the
paper owes much to the proofs for the branching particle system in
\citep{Adamczak:2011kx}. Our strategy is to take a proof of \citep{Adamczak:2011kx}
and control the dressing behavior in a suitable way. This is the main
technical difficulty of the paper. We will comment once again about
the strategy after presenting decomposition \eqref{eq:fundamentalDecomposition}. 

To make this paper self-contained we will now briefly present the
aspects of the backbone construction which are relevant to our paper.
Much of the text below is {}``borrowed'' from \citep[Section 2.4]{Berestycki:2011kx}.
We refer to this paper a reader interested in a more general and detailed
description%
\footnote{The author thanks Andreas Kyprianou for letting to use the parts of
description in \citep{Berestycki:2011kx}. We note that our notation
is mainly consistent with the one in \citep{Berestycki:2011kx}. The
only notable exception is of $\alpha$ in the branching mechanism
function $\psi$. We prefer to assume that $\alpha>0$ and put $-\alpha$
in \eqref{eq:branchingMechnism} instead of the form in \citep[Section 2.1]{Berestycki:2011kx}.%
}. Let us recall that we assume that the branching mechanism is given
by \eqref{eq:branchingMechnism} and that $\alpha>0$. Let $\lambda^{*}$
be the largest root of $\psi(\lambda)=0,$ i.e. 
\[
\lambda^{*}=\frac{\alpha}{\beta}.
\]
We denote $\psi^{*}(\lambda):=\psi(\lambda+\lambda^{*})$ and check
that 
\begin{equation}
\psi^{*}(\lambda)=\beta(\lambda+\frac{\alpha}{\beta})^{2}-\alpha(\lambda+\frac{\alpha}{\beta})=\beta\lambda^{2}+2\alpha\lambda+\frac{\alpha^{2}}{\beta}-\alpha\lambda-\frac{\alpha^{2}}{\beta}=\alpha\lambda+\beta\lambda^{2}.\label{eq:subcritical}
\end{equation}
The superprocess construction presented in Section \ref{sec:Preliminareis}
is also valid for $\psi^{*}$ being the branching mechanism. This
superprocess is subcritical i.e. its total mass decays exponentially
fast. We will refer to it using additional superscript $^{*}$, e.g.
$\ev{}_{\nu}^{*}$ \eqref{eq:logLaplace}. 

We calculate the branching law of the prolific backbone (see \citep[Section 2.4]{Berestycki:2011kx})
\begin{equation}
F(s)=\frac{1}{\lambda^{*}}\psi(\lambda^{*}(1-s))=\rbr{-\alpha(1-s)+\alpha(1-s)^{2}}=\alpha\rbr{-1+s+1-2s+s^{2}}=\alpha(s^{2}-s).\label{eq:prolificBranching}
\end{equation}
Let $\mathcal{M}_{a}(\R^{d})\subset\mathcal{M}_{F}(\R^{d})$ be the
space of finite atomic measures on $\R^{d}$. We shall write $\cbr{z_{t}}_{t\geq0}$
for a branching $\T{}$-motion whose total mass has generator given
by \eqref{eq:prolificBranching}. Hence $z$ is the $\mathcal{M}_{a}(\R^{d})$-valued
Markov process in which individuals from the moment of birth, live
for an independent and exponentially distributed period of time with
parameter $\psi'(\lambda^{*})=\alpha$ during which they execute an
Ornstein-Uhlenbeck diffusion issued from their position of birth and
at death they give birth at the same position to two offspring. Let
us stress that the backbone process does not suffer from extinction.
In our case the process $z$ is nothing else than the Ornstein-Uhlenbeck
branching process (studied in \citep{Adamczak:2011kx}). We shall
also refer to $z$ as the backbone (the name will become self-explanatory
soon). The initial configuration of $z$ is denoted by $\gamma\in\mathcal{M}_{a}(\R^{d})$.
Moreover, when referring to individuals in $z$ we may use the classical
Ulam-Harris notation, see for example \citep[p. 290]{Hardy:2006aa}.
The only feature that we really need of the Ulam-Harris notation is
that the individuals are uniquely identifiable amongst $\mathcal{T}$,
the set labels of individuals realized in $z$. For each individual
$u\in\mathcal{T}$ we shall write $\tau_{u}$ and $\sigma_{u}$ for
its birth and death times respectively and $\cbr{z_{u}(r):r\in[\tau_{u},\sigma_{u}]}$
for its spatial trajectory. 
\begin{defn}
\label{def:backboneConstruction}For $\nu\in\mathcal{M}_{{\rm a}}(\mathbb{R}^{d})$
and $\nu\in\mathcal{M}_{F}(\mathbb{R}^{d})$ let $z$ be the Ornstein-Uhlenbeck
branching process with initial configuration $\gamma$ and $\widetilde{X}$
an independent copy of $X$ under $\mathbb{P}_{\nu}^{*}$ (that is
with the subcritical branching mechanism function \eqref{eq:subcritical}).
Then we define a $\mathcal{M}_{F}(\mathbb{R}^{d})$-valued stochastic
process $\cbr{\Lambda_{t}}_{t\geq0}$ by 
\begin{equation}
\Lambda=\widetilde{X}+I^{\mathbb{N}^{*}},\label{eq:backboneConstruction}
\end{equation}
where the processes $\cbr{I^{\mathbb{N}^{*}}}_{t\geq0}$ is independent
of $\tilde{X}$. Moreover, this process is described path-wise as
follows (we note that the construction in \citep{Berestycki:2011kx}
contains more ingredients as it covers a larger class of branching
diffusions)
\begin{description}
\item [{Continuous~immigration}] The process $I^{\mathbb{N}^{*}}$ is
measure-valued on $\mathbb{R}^{d}$ such that 
\[
I_{t}^{\mathbb{N}^{*}}:=\sum_{u\in\mathcal{T}}\sum_{t\wedge\tau_{u}<r\leq t\wedge\sigma_{u}}X_{t-r}^{(1,u,r)},
\]
where, given $z$, independently for each $u\in\mathcal{T}$ such
that $\tau_{u}<t$, the processes $X_{\cdot}^{(1,u,r)}$ are countable
in number and correspond to $\mathcal{X}$-valued, Poissonian immigration
along the space-time trajectory $\{(z_{u}(r),r):r\in(\tau_{u},t\wedge\sigma_{u}]\}$
with rate $2\beta{\rm d}r\times{\rm d}\mathbb{N}_{z_{u}(r)}^{*}$.
To complete the definition we need to explain measures $\cbr{N_{x}^{*},x\in\R^{d}}$.
They are associated with the laws $\cbr{P_{\delta_{x}}^{*},x\in\R^{d}}$
defined on the same measurable space, namely
\begin{equation}
\mathbb{N}_{x}^{*}(1-e^{-\langle f,X_{t}\rangle})=-\log\mathbb{E}_{\delta_{x}}^{*}(e^{-\langle f,X_{t}\rangle}),\label{eq:dynkinKuznetzow}
\end{equation}
for all $f\in bp(\mathbb{R}^{d})$ and $t\geq0$. Such measures are
formally defined and explored in detail in \citep{Dynkin:2004fk}.
Intuitively speaking, the branching property implies that $\mathbb{P}_{\delta_{x}}^{*}$
is an infinitely divisible measure on the path space of $X$, $\mathcal{X}:=\mathcal{M}(\mathbb{R}^{d})\times[0,\infty)$,
and \eqref{eq:dynkinKuznetzow} is a `Lévy-Khinchine' formula in which
$\mathbb{N}_{x}^{*}$ plays the role of its `Lévy measure'. In this
sense, $\mathbb{N}_{x}^{*}$ can be considered as the `rate' at which
superprocesses `with zero initial mass' contribute to a unit mass
at position $x$. 
\end{description}
Moreover, we denote the law of $\Lambda$ by $\mathbf{P}_{\nu\times\gamma}$.
\end{defn}
We will now present the main result concerning the backbone construction.
First we randomize the law of $\mathbf{P}_{\nu\times\gamma}$ for
$\nu\in\mathcal{M}_{F}(\R^{d})$ by replacing the deterministic choice
of $\gamma$ with a Poisson random measure having intensity $\lambda^{*}|\nu|$.
We denote the resulting law by $\mathbf{P}_{\nu}$. We have \citep[Theorem 2]{Berestycki:2011kx}
\begin{thm}
\label{thm:backboneConstruction}For any $\nu\in\mathcal{M}_{F}(\mathbb{R}^{d})$,
the process $(\Lambda,\mathbf{P}_{\nu})$ is Markovian and has the
same law as $(X,\mathbb{P}_{\nu})$.
\end{thm}
The construction above states that the superprocess can be seen as
immigration of a mass, which further will be called dressing, in a
Poissonian fashion along the backbone $z$. The backbone is a supercritical
process contrary to the fact that the immigrating mass follows the
superprocess dynamics with subcritical mechanism $\psi^{*}$. This,
up to some degree, means that the backbone is what really matters
and the superprocess can be regarded as a discrete entity. Following
the fact z is the Ornstein-Uhlenbeck branching process is that many
of proof techniques from \citep{Adamczak:2011kx} can be reused.

The backbone cannot die thus the set of extinction has a particularly
simple description for $(\Lambda,\mathbf{P}_{\nu})$. Namely, 
\[
Ext=\cbr{z_{0}=0},
\]
that is, the extinction holds only when no backbone particle ever
appeared.

We denote two families of processes. The first $\cbr{D_{t}^{s}}_{t\geq0}$
where $s\geq0$ is given by

\begin{equation}
D_{t}^{s}:=\sum_{u\in\mathcal{T}}\sum_{t\wedge\tau_{u}<r\leq t\wedge\sigma_{u}}X_{t-r+s}^{(1,u,r)}1_{r\leq s}.\label{eq:dressing}
\end{equation}
More intuitively speaking this process describes the evolution of
the dressing which appeared in the system before time $s$. The complementary
process is defined in terms of the backbone. Namely, given the $i$-th
particle of $z_{s}$, by $\cbr{\Gamma_{t}^{i,s}}_{t\geq0}$ we denote
process 
\begin{equation}
\Gamma_{t}^{i,s}:=\sum_{u\in\mathcal{T}^{i,s}}\sum_{t\wedge\tau_{u}<r\leq t\wedge\sigma_{u}}X_{t-r+s}^{(1,u,r)},\label{eq:singleProlific}
\end{equation}
where $\mathcal{T}^{i,s}$ is a (random) tree stemming from $i$-th
particle. Intuitively $\Gamma^{i,s}$ is a {}``sub-superprocess''
stemming from the $i$-th prolific individual at time $s$. Let us
recall \eqref{eq:backboneConstruction} and fix $s\geq0$, for any
$t\geq s$ we have 
\begin{equation}
X_{t}=X_{t}^{0}+D_{t-s}^{s}+\sum_{i=1}^{|z_{t}|}\Gamma_{t-s}^{i,s}.\label{eq:fundamentalDecomposition}
\end{equation}
Now we come back to the description of the proof strategy.
\begin{rem}
\label{rem:proofStrategy}The first two terms of \eqref{eq:fundamentalDecomposition}
are subcritical superprocesses and as such are negligible when $t\gg s$.
The third term is a sum of random variables indexed with the branching
process $z$ to which some techniques similar to \citep{Adamczak:2011kx}
can be applied. 

The proofs in \citep{Adamczak:2011kx} relied on the existence of
an explicit coupling of two Ornstein-Uhlenbeck processes \citep[Fact 4.1]{Adamczak:2011kx}.
In \citep{Adamczak:2011kx} this coupling can be \textquotedblleft{}transferred\textquotedblright{}
on the level of branching processes in a way that two coupled systems
shared the same genealogical structure and the corresponding particles
were coupled OU processes. The main advantage of such approach is
that it makes proofs conceptually clear. However in this paper we
decided not to follow this strategy. Reasons are twofold. Firstly,
transferring coupling to the superprocess level is less obvious (although
possible), secondly in the proofs below we use analytical notions
which should be easier to use for more general diffusions. Let us
note that beside these changes {}``the high level structure'' of
\citep{Adamczak:2011kx}'s proofs could be reused. The main idea taken
from \citep{Adamczak:2011kx} is to study the system on two different
time scales.
\end{rem}

\subsection{Ornstein-Uhlenbeck semigroup properties}

Throughout the proofs we will denote 
\[
\tilde{f}:=f-\ddp f{\eq}.
\]
We will now summarize properties of the Ornstein-Uhlenbeck semigroup
(with the infinitesimal operator \eqref{eq:infinitesimalOP}) which
we will use later
\begin{fact}
\label{fac:decay1}Let $f\in\pspace(\R^{d})$ and $\tilde{f}$ is
defined as above. Then there exist constants $C,n>0$ such that for
any $t\geq0$ we have

\begin{equation}
|\T tf(x)|\leq C\rbr{\norm x{}ne^{-\mu t}+1}.\label{eq:triavialEstimation}
\end{equation}
\begin{equation}
|\T t\tilde{f}(x)|\leq C(1+\norm x{}n)e^{-\mu t},\quad\T t\tilde{f}(0)\leq Ce^{-2\mu t}.\label{eq:decayCentered}
\end{equation}
We also have
\begin{equation}
\T tid(x)=e^{-\mu t}id(x),\label{eq:semigroupFirstEignenvalue}
\end{equation}
where $id(x)=x$. Moreover 
\[
\lim_{t\rightarrow+\infty}e^{\mu t}\T t\tilde{f}(x)=x\circ\ddp{\grad f}{\eq},
\]
where $\circ$ is the standard scalar product and $\varphi$ is given
by \eqref{eq:equilibrium}. Finally 
\begin{equation}
|e^{\mu t}\T t\tilde{f}(x)-x\circ\ddp{\grad f}{\eq}|\leq C(1+\norm x{}n)e^{-\mu t}.\label{eq:decaySpeed}
\end{equation}
\end{fact}
\begin{proof}
As explained in \citep[(19)]{Adamczak:2011kx} we have $\T tf(x)=\ev{}f(xe^{-\mu t}+ou(t)G),$
where $G\sim\eq$ and $ou(t)=\sqrt{1-e^{-2\mu t}}$. Using the triangle
inequality and the binomial expansion one gets
\[
\T tf(x)=\ev{}f(xe^{-\mu t}+ou(t)G)\leq c\ev{}\norm{xe^{-\mu t}+ou(t)G}{}n\cleq\ev{}\rbr{\norm{xe^{-\mu t}}{}{}+\norm{ou(t)G}{}{}}^{n}\cleq\sum_{i=0}^{n}\norm{xe^{-\mu t}}{}i\ev{}\norm{ou(t)G}{}{n-i}.
\]
Now \eqref{eq:triavialEstimation} follows simply by the fact that
any moment of a Gaussian variable exists. The rest of the statements
were proved in \citep[Lemma 4.3]{Adamczak:2011kx}.
\end{proof}

\subsection{Moments calculation}

The main aim of this section is to calculate moments of $\ddp{\Gamma_{t}^{i,s}}f$
where $\Gamma_{t}^{i,s}$ is given by \eqref{eq:singleProlific}.
We will utilize moments up to order $4$ but, as it yields no additional
cost, we make some calculations for an arbitrary order.

Before the calculations we recall the generalization of the chain
rule, Faà di Bruno's formula, which states that
\begin{equation}
\frac{d^{k}}{dx^{k}}h(g(x))=\sum_{\mathbf{m}\in A_{k}}a_{\mathbf{m}}\cdot h^{(m_{1}+\cdots+m_{k})}(g(x))\cdot\prod_{j=1}^{k}\left(g^{(j)}(x)\right)^{m_{j}},\label{eq:diBrunoFormula}
\end{equation}
where $a_{m_{1},\ldots,m_{n}}:=\frac{k!}{m_{1}!\,1!^{m_{1}}\, m_{2}!\,2!^{m_{2}}\,\cdots\, m_{n}!\, n!^{m_{n}}}$
and the sum is over the set $A_{k}$ of all $k$-tuples of non-negative
integers $\mathbf{m}=(m_{1},\ldots,m_{k})$ satisfying the constraint
$\sum_{j=1}^{k}jm_{j}=k$. 

Let $f\in bp(\R^{d})$; we recall that $u_{f}$ is the solution of
\eqref{eq:integralEq}. We introduce an additional parameter $\theta>0$
and denote 
\[
\theta\mapsto u_{\theta f}(x,t)=\T t(\theta f)(x)-\int_{0}^{t}\T{t-s}\sbr{\psi(u_{\theta f}(\cdot,s))}(x)\dd s.
\]
It is obvious (as \eqref{eq:integralEq} has a unique non-negative
solution) that 
\[
u_{0f}(x,t)=0.
\]
Differentiating with respect to $\theta$ and using \eqref{eq:diBrunoFormula}
we get 
\[
\frac{\partial}{\partial\theta}u_{\theta f}(x,t)=\T tf(x)-\int_{0}^{t}\T{t-s}\sbr{\frac{\partial}{\partial\theta}u_{\theta f}(x,s)\psi'(u_{\theta f}(\cdot,s))}(x).
\]
\[
\frac{\partial^{k}}{\partial\theta^{k}}u_{\theta f}(x,t)=-\int_{0}^{t}\T{t-s}\sbr{\sum_{\mathbf{m}\in A_{k}}a_{\mathbf{m}}\psi^{(m_{1}+\cdots+m_{k})}(u_{\theta f}(x,s))\cdot\prod_{j=1}^{k}\left(\frac{\partial^{j}}{\partial\theta^{j}}u_{\theta f}(x,s)\right)^{m_{j}}}(x),\quad\text{for }k\geq2.
\]
The above calculation is formal. To legalize them we fix $T>0$ and
notice that by Fact\ref{fact:exponentailMomentsofTotalMass} there
exists $\theta_{T}>0$ such that $\ev{}e^{\theta\ddp f{X_{t}}}<+\infty$
for any $\theta\leq\theta_{T}$ and $t\leq T$. By the standard properties
of the Laplace transform we conclude that the derivatives $\frac{\partial^{k}}{\partial\theta^{k}}u_{\theta f}(x,t)$
exist for $\theta>-\theta_{T}$ and $t\leq T$ and are continuous
as function of $\theta$. Further, one need to apply standard calculus
tricks to conclude that the above formulas a valid for any $t\leq T$
and $\theta>-\theta_{T}$. Finally, for $\theta=0$ the formulas are
valid for any $t$.

We denote $u_{f}^{k}(x,t):=\left.\frac{\partial^{k}}{\partial\theta^{k}}u_{\theta f}(x,t)\right|_{\theta=0}$.
The above equations yield $u_{f}^{0}(x,t)=0$. Using \eqref{eq:branchingMechnism}
we obtain 
\[
u_{f}^{1}(x,t)=\T tf(x)+\alpha\int_{0}^{t}\T{t-s}\sbr{u_{f}^{1}(\cdot,s)}(x).
\]
\[
u_{f}^{k}(x,t)=-\int_{0}^{t}\T{t-s}\sbr{\sum_{\mathbf{m}\in A_{k}}a_{\mathbf{m}}\psi^{(m_{1}+\cdots+m_{k})}(0)\cdot\prod_{j=1}^{k}\left(u_{f}^{j}(x,s)\right)^{m_{j}}}(x),\quad\text{for }k\geq2.
\]
The first equation is solved by 
\begin{equation}
u_{f}^{1}(x,t)=\T t^{\alpha}f(x).\label{eq:superpocesMoment1}
\end{equation}
To treat the second one we denote $B_{k}:=A_{k}\backslash\cbr{(0,\ldots,0,1)}$
and notice that 
\[
u_{f}^{k}(x,t)=-\int_{0}^{t}\T{t-s}\sbr{-\alpha u_{f}^{1}(x,s)+\sum_{\mathbf{m}\in B_{k}}a_{\mathbf{m}}\psi^{(m_{1}+\cdots+m_{k})}(0)\cdot\prod_{j=1}^{k}\left(u_{f}^{j}(x,s)\right)^{m_{j}}}(x),\quad\text{for }k\geq2.
\]
It is solved by
\begin{equation}
u_{f}^{k}(x,t)=-\int_{0}^{t}\T{t-s}^{\alpha}\sbr{\sum_{\mathbf{m}\in B_{k}}a_{\mathbf{m}}\psi^{(m_{1}+\cdots+m_{k})}(0)\cdot\prod_{j=1}^{k}\left(u_{f}^{j}(x,s)\right)^{m_{j}}}(x),\quad\text{for }k\geq2.\label{eq:diBruno}
\end{equation}
The above equations makes it possible to calculate recursively any
$u_{f}^{k}$. Performing this operation for $k\leq4$ (which is all
we need in this paper) and taking into account the special form of
$\psi$ we get 
\begin{equation}
u_{f}^{2}(x,t)=-2\beta\intc t\T{t-s}^{\alpha}\sbr{\rbr{\T s^{\alpha}f(\cdot)}^{2}}(x)\dd s.\label{eq:superprocesMoment2}
\end{equation}
\begin{equation}
u_{f}^{3}(x,t)=-6\beta\int_{0}^{t}\T{t-s}^{\alpha}\sbr{u_{f}^{1}(\cdot,s)u_{f}^{2}(\cdot,s)}(x),\label{eq:superprocessMoment3}
\end{equation}
\begin{equation}
u_{f}^{4}(x,t)=-\beta\int_{0}^{t}\T{t-s}^{\alpha}\sbr{4u_{f}^{1}(\cdot,s)u_{f}^{3}(\cdot,s)+\rbr{u_{f}^{2}(\cdot,s)}^{2}}(x),\label{eq:superprocessMoment4}
\end{equation}
The same calculation are valid for the superprocess with branching
mechanism $\psi^{*}$ given by \eqref{eq:subcritical} ( which requires
only changing $\alpha$ to $-\alpha$). We will denote the quantities
corresponding to the system with $\psi^{*}$ using additional superscript
$^{*}$. 

Let us now prove some properties of $u_{f}^{k}$. 
\begin{fact}
\label{fac:misc1}Equations \eqref{eq:superpocesMoment1},\eqref{eq:superprocesMoment2}-\eqref{eq:superprocessMoment4}
are well-defined for any $f\in\pspace(\R^{d})$. Moreover, given $f\in\mathcal{C}(\R^{d})$,
there exist $C,n>0$ such that \textup{
\begin{equation}
|u_{f}^{2}(x,t)|\leq Ce^{2\alpha t}(\norm x{}n+1).\label{eq:tmp19}
\end{equation}
\begin{equation}
|u_{f}^{*,2}(x,t)|\leq Ce^{-\alpha t}(\norm x{}n+1).\label{eq:estimaiton77}
\end{equation}
We have also 
\begin{equation}
|u_{f}^{*,3}(x,t)|\leq Ce^{-\alpha t}(\norm x{}n+1),\quad|u^{*,4}(x,t)|\leq Ce^{-\alpha t}(\norm x{}n+1).\label{eq:tildeEstimate}
\end{equation}
}
\end{fact}
Let us note that potentially we can obtain better constants in some
estimation. E.g. {}``$n$'' in \eqref{eq:estimaiton77} could be
smaller than in \eqref{eq:tildeEstimate}. However exact constants
are not important in our proofs.
\begin{proof}
Formulas \eqref{eq:diBruno}-\eqref{eq:superprocessMoment4} were
obtained for $f\in bp(\R^{d})$ but can be {}``upgraded'' to $f\in\pspace$
by means of iterated application of standard integral-theoretic reasonings
and Fact \ref{fac:decay1}. \eqref{eq:tmp19} follows by Fact \ref{fac:decay1}
and the following calculations 
\[
|u_{f}^{2}(x,t)|\cleq\intc t\T{t-s}^{\alpha}\sbr{e^{2\alpha s}\rbr{\norm{\cdot}{}n+1}^{2}}(x)\cleq e^{\alpha t}\intc te^{\alpha s}(\norm x{}{2n}e^{-\mu(t-s)}+1)\dd s\cleq e^{2\alpha t}(\norm x{}{2n}+1),
\]
and in the same vein we obtain \eqref{eq:estimaiton77}. In order
to prove \eqref{eq:tildeEstimate} we utilize \eqref{eq:superprocessMoment3}
and Fact \ref{fac:decay1} 
\begin{multline*}
|u_{f}^{*,3}(x,t)|\leq\int_{0}^{t}\T{t-s}^{-\alpha}\sbr{\left|u_{f}^{*,1}(\cdot,s)u_{f}^{*,2}(\cdot,s)\right|}(x)\dd s\\
\cleq e^{-\alpha t}\int_{0}^{t}e^{\alpha s}\T{t-s}\sbr{e^{-\alpha s}(1+\norm{\cdot}{}n)e^{-\alpha s}(1+\norm{\cdot}{}n)}(x)\dd s\cleq e^{-\alpha t}(1+\norm{\cdot}{}{n_{1}}),
\end{multline*}
for some $n_{1}\geq0$. The case of $u_{f}^{4}(x,t)$ follows similarly.
By \eqref{eq:superprocessMoment4} we have
\begin{multline*}
|u_{f}^{*,4}(x,t)|\cleq\int_{0}^{t}\T{t-s}^{-\alpha}\sbr{\left|u_{f}^{*,1}(\cdot,s)u_{f}^{*,3}(\cdot,s)\right|+\left|u_{f}^{*,2}(\cdot,s)\right|^{2}}(x)\dd s\\
\cleq e^{-\alpha t}\int_{0}^{t}e^{-\alpha s}\T{t-s}\sbr{(1+\norm{\cdot}{}{3n})}(x)\dd s\cleq e^{-\alpha t}(1+\norm x{}{n_{2}}),
\end{multline*}
 for some $n_{2}\geq0$. 
\end{proof}
Now let us recall to the backbone construction given in Definition
\ref{def:backboneConstruction}. We fix $f\in bp(\R^{d})$ and apply
\citep[Theorem 1]{Berestycki:2011kx} with $\mu=0,\nu=\delta_{x},f=\theta f,h=0$.
This yields
\[
\mathbf{E}_{0\times\delta_{x}}\rbr{e^{-\langle\theta f,\Lambda_{t}\rangle}}=V_{\theta f}(x,t),
\]
where $V_{\theta f}(x,t)$ is the unique $[0,1]$-valued solution
to the integral equation
\[
V_{\theta f}(x,t)=1+\frac{\beta}{\alpha}\int_{0}^{t}\T{t-s}\sbr{\psi^{*}\rbr{-\frac{\alpha}{\beta}V_{\theta f}(\cdot,s)+u_{\theta f}^{*}(\cdot,s)}-\psi^{*}(u_{\theta f}^{*}(\cdot,s))}(x){\rm d}s,
\]
(we recall that $u^{*}$ is the solution of \eqref{eq:integralEq}
with the branching mechanism $\psi^{*}$). Obviously, we have 
\[
V_{0}(x,t)=1.
\]
Using \eqref{eq:diBrunoFormula} we obtain (one has to justify the
validity of the below calculations in the same spirit as for $u_{\theta f}$.
We skip some arguments to make expressions more clear)
\begin{multline*}
\frac{\partial^{k}V_{\theta f}}{\partial\theta^{k}}=\frac{\beta}{\alpha}\int_{0}^{t}\T{t-s}\left[\sum_{\mathbf{m}\in A_{k}}a_{\mathbf{m}}\psi^{*(m_{1}+\ldots+m_{k})}(-\frac{\alpha}{\beta}V_{\theta f}+u_{\theta f}^{*})\prod_{j=1}^{k}\rbr{-\frac{\alpha}{\beta}\frac{\partial^{j}V_{\theta f}}{\partial\theta^{j}}+\frac{\partial^{j}u_{\theta f}^{*}}{\partial\theta^{j}}}^{m_{j}}\right.\\
\left.-\sum_{\mathbf{m}\in A_{k}}a_{\mathbf{m}}\psi^{*(m_{1}+\ldots+m_{k})}(u_{\theta f}^{*})\prod_{j=1}^{k}\rbr{\frac{\partial^{j}u_{\theta f}^{*}}{\partial\theta^{j}}}^{m_{j}}\right](x){\rm d}s,\quad k\geq1,
\end{multline*}
We denote $V_{f}^{k}:=\left.\frac{\partial^{k}V_{\theta f,0}}{\partial\theta^{k}}\right|_{\theta=0}$
. The above equation yields
\begin{multline*}
V_{f}^{k}(x,t)=\frac{\beta}{\alpha}\int_{0}^{t}\T{t-s}\left[\sum_{\mathbf{m}\in A_{k}}a_{\mathbf{m}}\psi^{*(m_{1}+\ldots+m_{k})}(-\alpha/\beta)\prod_{j=1}^{k}\rbr{-\frac{\alpha}{\beta}V_{f}^{j}+u_{f}^{*,j}}^{m_{j}}\right.\\
\left.-\sum_{\mathbf{m}\in A_{k}}a_{\mathbf{m}}\psi^{*(m_{1}+\ldots+m_{n})}(0)\prod_{j=1}^{k}\rbr{u_{f}^{*,j}}^{m_{j}}\right](x){\rm d}s.
\end{multline*}
We recall that $B_{k}=A_{k}\backslash\cbr{(0,\ldots,0,1)}$ and $u_{0}^{*}=0,V_{0}=1$.
Moreover $\star{\psi}'(-\frac{\alpha}{\beta})=-2\beta\frac{\alpha}{\beta}+\alpha=-\alpha$
and $\star{\psi}'(0)=\alpha$. Therefore
\begin{multline*}
V_{f}^{k}(x,t)=\frac{\beta}{\alpha}\int_{0}^{t}\T{t-s}\left[-\alpha\rbr{-\frac{\alpha}{\beta}V_{f}^{k}+u_{f}^{*,k}}+\sum_{\mathbf{m}\in B_{k}}a_{\mathbf{m}}\psi^{*(m_{1}+\ldots+m_{k})}(-\alpha/\beta)\prod_{j=1}^{k}\rbr{-\frac{\alpha}{\beta}V_{f}^{j}+u_{f}^{*,j}}^{m_{j}}\right.\\
\left.-\alpha u_{f}^{*,k}-\sum_{\mathbf{m}\in B_{k}}a_{\mathbf{m}}\psi^{*(m_{1}+\ldots+m_{k})}(0)\prod_{j=1}^{k}\rbr{u_{f}^{*,j}}^{m_{j}}\right](x){\rm d}s.
\end{multline*}
This equation is solved by
\begin{multline}
V_{f}^{k}(x,t)=\frac{\beta}{\alpha}\int_{0}^{t}\T{t-s}^{\alpha}\left[\sum_{\mathbf{m}\in B_{k}}a_{\mathbf{m}}\psi^{*(m_{1}+\ldots+m_{k})}(-\alpha/\beta)\prod_{j=1}^{k}\rbr{-\frac{\alpha}{\beta}V_{f}^{j}+u_{f}^{*,j}}^{m_{j}}\right.\\
\left.-2\alpha u_{f}^{*,k}-\sum_{\mathbf{m}\in B_{k}}a_{\mathbf{m}}\psi^{*(m_{1}+\ldots+m_{k})}(0)\prod_{j=1}^{k}\rbr{u_{f}^{*,j}}^{m_{j}}\right](x){\rm d}s.\label{eq:generalEquation}
\end{multline}
We list now some properties of $V_{f}^{k}$ used in the proofs below
below. 
\begin{fact}
\textup{\label{fac:backBoneFacts}For any $f\in\pspace(\R^{d})$ we
have 
\begin{equation}
\mathbf{E}_{0\times\delta_{x}}\langle f,\Lambda_{t}\rangle^{k}=(-1)^{k}V_{f}^{k}(x,t),\quad k\in\mathbb{N},\label{eq:backBoneMoments}
\end{equation}
}
\begin{equation}
V_{f}^{1}(x,t)=-2\frac{\beta}{\alpha}\T tf(x)\sinh(\alpha t),\label{eq:mean}
\end{equation}
\begin{equation}
V_{f}^{2}(x,t)=\frac{\beta}{\alpha}\int_{0}^{t}\T{t-s}^{\alpha}\sbr{2\beta\rbr{\T s^{\alpha}f(\cdot)}^{2}-2\beta\rbr{u_{f}^{*,1}(\cdot,s)}^{2}-2\alpha u_{f}^{*,2}(\cdot,s)}(x)\dd s.\label{eq:secondMoment}
\end{equation}
Moreover, we have
\begin{equation}
V_{f}^{2}(x,t)\leq Ce^{2\alpha t}(\norm x{}{2n}+1).\label{eq:estimationAAA}
\end{equation}
\end{fact}
\begin{proof}
Analogously as in the proof in Fact \ref{fac:misc1} one can argue
that the formulas proved for $f\in bp(\R^{d})$ could be also used
to $f\in\pspace$ . Obtaining \eqref{eq:backBoneMoments} is completely
standard. \eqref{eq:mean} can be obtained by a simple calculation,
namely using \eqref{eq:generalEquation} and \eqref{eq:diBruno} 
\[
V_{f}^{1}(x,t)=-2\beta\int_{0}^{t}\T{t-s}^{\alpha}\sbr{u_{f}^{*,1}(\cdot,s)}(x)\dd s.
\]
By \eqref{eq:superpocesMoment1} it follows 
\[
V_{f}^{1}(x,t)=-2\beta\int_{0}^{t}e^{\alpha(t-s)}\T{t-s}\sbr{e^{-\alpha s}\T sf(\cdot)}(x)\dd s=-2\beta e^{\alpha t}\T tf(x)\intc te^{-2\alpha s}\dd s=-2\frac{\beta}{\alpha}\T tf(x)\sinh(\alpha t).
\]
To prove \eqref{eq:secondMoment} we again use\eqref{eq:generalEquation}
\[
V_{f}^{2}(x,t)=\frac{\beta}{\alpha}\int_{0}^{t}\T{t-s}^{\alpha}\sbr{2\beta\rbr{-\frac{\alpha}{\beta}V_{f}^{1}(\cdot,s)+u_{f}^{*,1}(\cdot,s)}^{2}-2\beta\rbr{u_{f}^{*,1}(\cdot,s)}^{2}-2\alpha u_{f}^{*,2}(\cdot,s)}(x)\dd s.
\]
We notice that $-\frac{\alpha}{\beta}V_{f}^{1}(x,s)+u_{f}^{*,1}(x,s)=\T s^{\alpha}f(x)$
which is enough to prove \eqref{eq:secondMoment}. \eqref{eq:estimationAAA}
holds by \eqref{eq:estimaiton77}, \eqref{eq:triavialEstimation}
and the following calculation
\begin{multline*}
V_{f}^{2}(x,t)\cleq\int_{0}^{t}\T{t-s}^{\alpha}\sbr{\rbr{e^{\alpha s}\T sf(\cdot)}^{2}+\rbr{e^{-\alpha s}\T sf(\cdot)}^{2}+|u_{f}^{*,2}(\cdot,s)|}(x)\dd s\\
\cleq e^{\alpha t}\int_{0}^{t}e^{\alpha s}\T{t-s}\sbr{(\norm{\cdot}{}n+1)^{2}}(x)\dd s\cleq e^{2\alpha t}(\norm x{}{2n}+1).
\end{multline*}

\end{proof}

\subsection{Dressing behavior}

Let us recall {}``the dressing process'' $\cbr{D_{t}^{s}}_{t\geq0}$
defined by \eqref{eq:dressing}. It is a superprocess which at time
$t=0$ is distributed as $X_{s}$ and then evolve according to subcritical
dynamics with the branching mechanism $\psi^{*}$. Using the Markov
property in Theorem \ref{thm:backboneConstruction} we obtain
\begin{cor}
\label{cor:dressing}Let $f\in bp(\R^{d})$ then
\[
\mathbb{E}_{\nu}(e^{-\langle f,D_{t}^{s}\rangle})=\ev{}_{\nu}\ev{}_{X_{s}}^{*}(e^{-\ddp f{X_{t}^{*}}}=\ev{}_{\nu}\exp\cbr{-\int_{\mathbb{R}^{d}}u_{f}^{*}(x,t)X_{s}({\rm d}x)}=\exp\cbr{-\int_{\R^{d}}v_{f}(x,s;t)\nu(\dd x)},
\]
where $u_{f}^{*}$ is defined as \eqref{eq:integralEquation} with
$\psi^{*}$ instead of $\psi$ and $v_{f}(x,s;t)$ is the solution
of 
\[
v_{f}(x,t;s)=\T t(u_{f}^{*}(\cdot,s))(x)-\beta\int_{0}^{t}\T u^{\alpha}[v_{f}(\cdot,t-u;s)^{2}](x)\dd u.
\]
\end{cor}
\begin{fact}
Let $f\in\pspace(\R^{d})$ then
\end{fact}
\begin{equation}
\mathbb{E}_{\nu}\langle f,D_{t}^{s}\rangle=e^{\alpha(s-t)}\int_{\R^{d}}\T{t+s}f(x)\nu(\dd{x)}.\label{eq:tmp17}
\end{equation}
The fact follows immediately by Corollary \ref{cor:dressing} and
Laplace transform calculations as in previous sections.

\subsection{Martingales}

We define two martingales $\cbr{W_{t}}_{t\geq0},\cbr{I_{t}}_{t\ge0}$,
associated with the backbone process $z$. Namely,
\[
W_{t}:=e^{-\alpha t}|z_{t}|,
\]
\[
I_{t}:=e^{-(\alpha-\mu)t}\sum_{i=1}^{|z_{t}|}z_{t}(i).
\]
They are closely related to $V$ and $H$ (defined by \eqref{eq:martingaleDef}
and \eqref{eq:secondMartingale}). Let us assume that all of them
are defined in terms of $\Lambda$ (see Theorem \eqref{thm:backboneConstruction}),
so that they {}``live'' in the same probability space.
\begin{fact}
\label{fac:martingaleEquvalence}$W$ is an $L^{2}$-bounded martingale.
We denote its limit by $W_{\infty}$ . Moreover,
\begin{equation}
V_{\infty}=\frac{\beta}{\alpha}W_{\infty}\quad\text{a.s}.\label{eq:martingaleEquiv1}
\end{equation}
$I$ is a martingale, which for $\alpha>2\mu$ it is $L^{2}$-bounded.
Then we denote its limit by $I_{\infty}$ and we have
\begin{equation}
H_{\infty}=\frac{\beta}{\alpha}I_{\infty}\quad\text{a.s}.\label{eq:martingaleEquiv2}
\end{equation}

\end{fact}
The proof uses some facts which are presented later on, hence it is
postponed to Appendix. 

We are now able do describe the relation of $V_{\infty},H_{\infty}$
and the dependence of the latter on the starting conditions. Let us
first denote by $\cbr{J_{i}}_{i\geq1}$ an i.i.d. sequence of random
variables distributed as $I_{\infty}$ but with assumption that $z_{0}=\delta_{0}$.
Under assumption $\alpha>2\mu$ this exists, we refer to \citep[Fact 3.8]{Adamczak:2011kx}
where it is know under name $H_{\infty}$ and to \citep[Remark 3.12]{Adamczak:2011kx}
for some information about its law. 
\begin{thm}
\label{fact:Hdependence}Let $\cbr{X_{t}}_{t\geq0}$ be the OU-superprocess
starting from $\nu\in\fMeas$ and $\alpha>2\mu$. Let us define 
\[
\hat{H}_{\infty}:=\frac{\beta}{\alpha}\rbr{\sum_{i=1}^{N}J_{i}+\sum_{i=1}^{N}x_{i}E_{i}},\quad\hat{V}_{\infty}:=\frac{\beta}{\alpha}\rbr{\sum_{i=1}^{N}x_{i}E_{i}},
\]
where $(x_{1},x_{2},\ldots,x_{N})$ is a realization of a Poisson
point process with intensity $\nu$ and $\cbr{E_{i}}_{i\geq1}$ is
an i.i.d. sequence of exponential random variables with parameter
$1$. Moreover, we assume that all defining objects are independent.
Then 
\[
(H_{\infty},V_{\infty})=^{d}(\hat{H}_{\infty},\hat{V}_{\infty}).
\]

\end{thm}
The proof follows easily using techniques of Fact \ref{fac:martingaleEquvalence}
and \citep[Fact 3.8]{Adamczak:2011kx}, hence is skipped.

\section{Proof of Theorem \ref{thm:smallBranching}\label{sec:smallBranching}}

We first gather properties of $V_{f}^{k}$, extending the list in
Fact \ref{fac:backBoneFacts}, in the case when $\alpha<2\mu$. We
recall that $\tilde{f}(x)=f(x)-\ddp f{\eq}$. 
\begin{fact}
\label{fac:slowBranchingMoments}Let $\alpha<2\mu$ and $f\in\pspace(\R^{d})$.
Then for any $x\in\R^{d}$ there is 
\begin{equation}
e^{-(\alpha/2)t}V_{\tilde{f}}^{1}(x,t)\conv0,\quad\text{as }t\conv+\infty,\label{eq:smallBranchingMoment1}
\end{equation}
\textup{
\begin{equation}
e^{-\alpha t}V_{\tilde{f}}^{2}(x,t)\conv\sigma_{f}^{2},\quad\text{as }t\conv+\infty,\label{eq:smallBranchingMoment2}
\end{equation}
as $t\conv+\infty$, where $\sigma_{f}^{2}$ is given by \eqref{eq:sigmaDefinition}.
We have $\sigma_{f}^{2}<+\infty$. Moreover, there exists $l>1/2$
such that 
\begin{equation}
\sup_{\norm x{}{}\leq t^{l}}|e^{-\alpha t}V_{\tilde{f}}^{2}(x,t)-\sigma_{f}^{2}|\conv0,\quad\text{as }t\conv+\infty.\label{eq:uniformConvergence}
\end{equation}
Finally, }there exist $C,n>0$ such that 
\begin{equation}
|V_{\tilde{f}}^{4}(x,t)|\leq Ce^{2\alpha t}(\norm x{}n+1).\label{eq:fourthMomentEstimate}
\end{equation}
\end{fact}
\begin{proof}
The proof of \eqref{eq:smallBranchingMoment1} follows easily by \eqref{eq:mean}
and Fact \ref{fac:decay1} hence is skipped. To get \eqref{eq:martingaleEquiv2}
we use \eqref{eq:secondMoment}and write 
\begin{multline}
e^{-\alpha t}V_{\tilde{f}}^{2}(x,t)=\frac{\beta}{\alpha}\int_{0}^{t}e^{-\alpha s}\T{t-s}\sbr{2\beta\rbr{\T s^{\alpha}\tilde{f}(\cdot)}^{2}}(x)\dd s\\
-2\frac{\beta}{\alpha}\int_{0}^{t}e^{-\alpha s}\T{t-s}\sbr{\beta\rbr{u_{\tilde{f}}^{*,1}(\cdot,s)}^{2}+\alpha u_{\tilde{f}}^{*,2}(\cdot,s)}(x)\dd s=:I_{1}(t)+I_{2}(t).\label{eq:tmp19-1}
\end{multline}
We have 
\[
I_{1}(t)=\frac{\beta}{\alpha}\int_{0}^{t}\T{t-s}\sbr{2\beta\rbr{e^{(\alpha/2)s}\T s\tilde{f}(\cdot)}^{2}}(x)\dd s.
\]
By \eqref{eq:decayCentered} the integrand in the last expression
can be estimated as follows
\begin{equation}
\T{t-s}\sbr{\rbr{e^{(\alpha/2)s}\T s\tilde{f}(\cdot)}^{2}}(x)\cleq e^{(\alpha-2\mu)s}\T{t-s}\sbr{(1+\norm{\cdot}{}n)^{2}}(x).\label{eq:tmpEst7}
\end{equation}
Using \eqref{eq:triavialEstimation} it can be checked that for any
$t\geq0$ we have $\T t\sbr{(1+\norm{\cdot}{}n)^{2}}(x)\cleq(1+\norm x{}{2n})$.
The dominated Lebesgue theorem implies that
\[
I_{1}(t)\conv\frac{\beta}{\alpha}\int_{0}^{\infty}\ddp{\eq}{2\beta\rbr{e^{(\alpha/2)s}\T s\tilde{f}(\cdot)}^{2}}\dd s<+\infty.
\]
A completely analogous argument, using \eqref{eq:superpocesMoment1}
and \eqref{eq:estimaiton77}, can be applied to treat $I_{2}(t).$

We will prove \eqref{eq:uniformConvergence} with $l=1$ (in the formulation
of the fact we used general $l$ in order to keep it consistent with
forthcoming Fact \ref{fac:criticalBranchingMoments}). It is enough
to show 
\[
\sup_{\norm x{}{}\leq t}e^{-\alpha t}|V_{\tilde{f}}^{2}(x,t)-V_{\tilde{f}}^{2}(0,t)|\conv0.
\]
We will prove this for the first term of \eqref{eq:tmp19-1} which
is the hardest one and leave other terms to the reader. That is, we
will show that 
\[
\sup_{\norm x{}{}\leq t}\left|\int_{0}^{t}e^{-\alpha s}\T{t-s}\sbr{2\beta\rbr{\T s^{\alpha}\tilde{f}(\cdot)}^{2}}(x)\dd s-\int_{0}^{t}e^{-\alpha s}\T{t-s}\sbr{2\beta\rbr{\T s^{\alpha}\tilde{f}(\cdot)}^{2}}(0)\dd s\right|\conv0.
\]
By Fact \ref{fac:decay1} we have
\begin{multline*}
\sup_{\norm x{}{}\leq t}\left|\int_{t/2}^{t}e^{-\alpha s}\T{t-s}\sbr{2\beta\rbr{\T s^{\alpha}\tilde{f}(\cdot)}^{2}}(x)\dd s\right|\cleq\sup_{\norm x{}{}\leq t}\left|\int_{t/2}^{t}e^{(\alpha-2\mu)s}\T{t-s}\sbr{\rbr{(1+\norm x{}n)}^{2}}(x)\dd s\right|\\
\cleq t^{2n}e^{((\alpha-2\mu)/2)t}\conv0.
\end{multline*}
It is well-known that $\T t\tilde{f}(x)=\int_{\R^{d}}g_{t}(xe^{-\mu t}+y)\tilde{f}(y)\dd y$,
where $g_{t}$ is the density of $\mathcal{N}(0,(1-e^{-2\mu t}))$
(see e.g. \citep[(19)]{Adamczak:2011kx}). Therefore using Fact \ref{fac:decay1}
and properties of the Gaussian density we obtain 
\begin{multline}
\sup_{\norm x{}{}\leq t}\left|\int_{0}^{t/2}e^{-\alpha s}\T{t-s}\sbr{2\beta\rbr{\T s^{\alpha}\tilde{f}(\cdot)}^{2}}(x)\dd s-\int_{0}^{t/2}e^{-\alpha s}\T{t-s}\sbr{2\beta\rbr{\T s^{\alpha}\tilde{f}(\cdot)}^{2}}(0)\dd s\right|\\
=2\beta\sup_{\norm x{}{}\leq t}\left|\int_{0}^{t/2}\int_{\R^{d}}e^{-\alpha s}\rbr{g_{t-s}(xe^{-\mu(t-s)}+y)-g_{t-s}(y)}\rbr{\T s^{\alpha}\tilde{f}(y)}^{2}\dd y\dd s\right|\\
\cleq\sup_{\norm x{}{}\leq t}\int_{0}^{t/2}\int_{\R^{d}}e^{-\alpha s}e^{-\mu(t-s)}\norm x{}{}(\norm y{}{}+t)g_{t-s}(y)\rbr{e^{(\alpha-\mu)s}(1+\norm y{}n)}^{2}\dd y\dd s\\
\cleq t^{2}\int_{0}^{t/2}e^{-\alpha s}e^{-\mu(t-s)}e^{2(\alpha-\mu)s}\dd s\leq t^{2}e^{-(\mu/2)t}\int_{0}^{t/2}e^{(\alpha-2\mu)s}\dd s\cleq t^{2}e^{-(\mu/2)t}\conv0.\label{eq:kurnik}
\end{multline}
In order to prove \eqref{eq:fourthMomentEstimate} we apply the triangle
inequality to \eqref{eq:generalEquation}
\begin{multline}
|V_{\tilde{f}}^{k}(x,t)|\cleq\sum_{\mathbf{m}\in B_{k}}\int_{0}^{t}\T{t-s}^{\alpha}\sbr{\prod_{j=1}^{k}\left|V_{\tilde{f}}^{j}+u_{\tilde{f}}^{*,j}\right|^{m_{j}}}(x)\dd s\\
+\int_{0}^{t}\T{t-s}^{\alpha}\sbr{u_{\tilde{f}}^{*,k}}(x)\dd s+\sum_{\mathbf{m}\in B_{k}}\int_{0}^{t}\T{t-s}^{\alpha}\sbr{\prod_{j=1}^{k}\left|u_{\tilde{f}}^{*,j}\right|^{m_{j}}}(x)\dd s.\\
\label{eq:generalEquation-1}
\end{multline}
By \eqref{eq:tildeEstimate} we can see that terms containing $u_{\tilde{f}}^{*,k}$
will not contribute (in fact they will be of order $(1+\norm x{}n)$
for some $n\in\mathbb{N}$, which is strictly smaller then the whole
expression). Therefore we skip these terms. For $k=3$ the first sum
reduces to (see also \eqref{eq:superprocessMoment3} for similar calculations). 

\begin{multline*}
\int_{0}^{t}\T{t-s}^{\alpha}\sbr{\left|V_{\tilde{f}}^{2}(\cdot,s)V_{\tilde{f}}^{1}(\cdot,s)\right|}(x)\dd s\leq e^{\alpha t}\intc t\T{t-s}e^{-\alpha s}\sbr{\left|V_{\tilde{f}}^{2}(\cdot,s)e^{(\alpha-\mu)s}(1+\norm{\cdot}{}n)\right|}(x)\dd s\\
\leq e^{\alpha t}\intc t\T{t-s}e^{-\mu s}\sbr{\left|e^{\alpha s}(1+\norm{\cdot}{}{2n})(1+\norm{\cdot}{}n)\right|}(x)\dd s\leq e^{(3\alpha/2)t}(1+\norm{\cdot}{}{3n}),
\end{multline*}
where we used \eqref{eq:mean}, \eqref{eq:tmpEst7}, and $\alpha<2\mu$.
In this way we have proved that $|V_{\tilde{f}}^{3}(x,t)|\leq e^{(3\alpha/2)t}(1+\norm{\cdot}{}n)$. 

Now we proceed to $k=4$. Again the second and third term of \eqref{eq:generalEquation-1}
can be skipped. The sum in the first term reduces to (see also \eqref{eq:superprocessMoment4}
for similar calculations) 
\begin{multline*}
\int_{0}^{t}\T{t-s}^{\alpha}\sbr{\left|V_{\tilde{f}}^{3}(\cdot,s)V_{\tilde{f}}^{1}(\cdot,s)\right|}(x)\dd s+\int_{0}^{t}\T{t-s}^{\alpha}\sbr{\left|V_{\tilde{f}}^{2}(\cdot,s)\right|^{2}}(x)\dd s\\
\leq\int_{0}^{t}\T{t-s}^{\alpha}\sbr{\left|e^{(3\alpha/2)s}(1+\norm{\cdot}{}{3n})e^{(\alpha-\mu)s}(1+\norm{\cdot}{}{3n})\right|}(x)\dd s\\
+\int_{0}^{t}\T{t-s}^{\alpha}\sbr{\left|e^{\alpha s}(1+\norm{\cdot}{}{2n})\right|^{2}}(x)\dd s\leq Ce^{2\alpha t}(1+\norm{\cdot}{}{4n}),
\end{multline*}
where we used the results proved above, Fact \ref{fac:decay1} and
$\alpha<2\mu$. 
\end{proof}
We now proceed to the proof of Theorem \ref{thm:smallBranching}.
In the proof we show a weak convergence hence we will work with $\Lambda$
which has the backbone representation as described in Definition \ref{def:backboneConstruction}.
We start with the following random vector 
\[
Z_{1}(t):=\rbr{e^{-\alpha t}|\Lambda_{t}|,e^{-(\alpha/2)t}(|\Lambda_{t}|-e^{\alpha t}V_{\infty}),e^{-(\alpha/2)t}\ddp{\Lambda_{t}}{\tilde{f}}}.
\]
Let $n\in\mathbb{N}$ to be fixed later; the limit of $Z_{1}(nt)$
is obviously the same as the on of $Z_{1}(t)$. By Fact \ref{fac:totalMassMartingale}
we have $V_{nt}-V_{t}\conv0$ a.s. Therefore the limit of $Z_{1}(t)$
coincides with the limit of 
\[
Z_{2}(t):=\rbr{e^{-\alpha t}|\Lambda_{t}|,e^{-(n\alpha/2)t}(|\Lambda_{nt}|-e^{n\alpha t}V_{\infty}),e^{-(n\alpha/2)t}\ddp{\Lambda_{nt}}{\tilde{f}}}.
\]
It is well known, by the so-called branching property, that the evolution
of the total mass of the system after time $nt$ is the same as it
was split into $\lfloor|\Lambda_{nt}|\rfloor$ of systems $\cbr{\Lambda_{t}^{i}}_{t\geq0}$
having initial mass $1$ and one system $\hat{\Lambda}$ of size $|\Lambda_{nt}|-\lfloor|\Lambda_{nt}|\rfloor$.
For each $i$ we can define a corresponding martingale by formula
\eqref{eq:martingaleDef} and denote its limit by $V_{\infty}^{i}$.
Moreover the limit of martingale of the system $\hat{\Lambda}$ is
denoted by $\hat{V}_{\infty}$. Obviously 
\[
V_{\infty}=e^{-n\alpha t}\rbr{\sum_{i=1}^{\lfloor|\Lambda_{nt}|\rfloor}V_{\infty}^{i}+\hat{V}_{\infty}}.
\]
Therefore, we have 
\[
e^{-(n\alpha/2)t}(|\Lambda_{nt}|-e^{n\alpha t}V_{\infty})=^{d}e^{-(n\alpha/2)t}\sum_{i=1}^{\lfloor|\Lambda_{nt}|\rfloor}(1-V_{\infty}^{i})+e^{-(n\alpha/2)t}(|\Lambda_{nt}|-\lfloor|\Lambda_{nt}|\rfloor-\hat{V}_{\infty}).
\]
 One easily sees that the second term converges to $0$ (in probability)
hence is negligible in our analysis. 

Now we perform arguably the most crucial step of the proof i.e. we
use \eqref{eq:fundamentalDecomposition}

\begin{equation}
e^{-(n\alpha/2)t}\ddp{\Lambda_{nt}}{\tilde{f}}=e^{-(n\alpha/2)t}\ddp{\tilde{X}_{nt}}{\tilde{f}}+e^{-(n\alpha/2)t}\ddp{D_{(n-1)t}^{t}}{\tilde{f}}+e^{-(n\alpha/2)t}\sum_{i=1}^{|z_{t}|}\ddp{\Gamma_{(n-1)t}^{i,t}}{\tilde{f}}.\label{eq:tmp185}
\end{equation}
Using \eqref{eq:tmp17} and Fact \ref{fac:decay1} we obtain 
\[
\ev{}_{\nu}e^{-(n\alpha/2)t}|\ddp{D_{(n-1)t}^{t}}{\tilde{f}}|\leq e^{-(n\alpha/2)t}\ev{}_{\nu}\ddp{D_{(n-1)t}^{t}}{|\tilde{f}|}=e^{-(n\alpha/2)t}e^{-(n-2)\alpha t}\int_{\R^{d}}\T t|\tilde{f}|(x)\nu(\dd x)\conv0,
\]
if $n\geq2$ is large enough. Analogously one can prove that $\ev{}e^{-(n\alpha/2)t}|\ddp{\tilde{X}_{nt}}{\tilde{f}}|\conv0$.
Using the facts proved above we conclude that the limit of $Z_{2}(t)$
is the same as the one of 
\[
Z_{3}(t):=\rbr{e^{-\alpha t}|\Lambda_{t}|,e^{(n\alpha/2)t}\sum_{i=1}^{\lfloor|\Lambda_{nt}|\rfloor}(1-V_{\infty}^{i}),e^{-(\alpha/2)t}\sum_{i=1}^{|z_{t}|}e^{-((n-1)\alpha/2)t}\ddp{\Gamma_{(n-1)t}^{i,t}}{\tilde{f}}}.
\]
We denote $Z_{t}^{n,i}:=e^{-((n-1)\alpha/2)t}\ddp{\Gamma_{(n-1)t}^{i,t}}{\tilde{f}}$
and $z_{t}^{n,i}:=\ev{}\rbr{Z_{t}^{n,i}|z_{t}}$. By Fact \ref{fac:backBoneFacts}
and Fact \ref{fac:decay1} we have
\[
|z_{t}^{n,i}|=ce^{-((n-1)\alpha/2)t}\sinh((n-1)\alpha t)|\T{(n-1)t}\tilde{f}(z_{t}(i))|\cleq e^{(\alpha(n-1)/2)t}e^{-\mu(n-1)t}(1+\norm{z_{t}(i)}{}k),
\]
for some $c>0,k>0$. Therefore by \citep[(23)]{Adamczak:2011kx} we
have 
\[
\ev{}\rbr{e^{-(\alpha/2)t}\sum_{i=1}^{|z_{t}|}|z_{t}^{n,i}|}\cleq e^{-(\alpha/2)t}e^{\alpha((n-1)/2)t}e^{-\mu(n-1)t}\conv0,\quad\text{as }t\conv+\infty,
\]
if we take $n$ sufficiently large (from this moment on we keep $n$
fixed). Therefore the limit of $Z_{3}$ is the same as the one of
\[
Z_{4}(t):=\rbr{e^{-\alpha t}|\Lambda_{t}|,e^{(n\alpha/2)t}\sum_{i=1}^{\lfloor|\Lambda_{nt}|\rfloor}(1-V_{\infty}^{i}),e^{-(\alpha/2)t}\sum_{i=1}^{|z_{t}|}(Z_{t}^{n,i}-z_{t}^{n,i})}.
\]
Let $l$ be the same as in \eqref{eq:uniformConvergence}. We check
that 
\begin{multline*}
\ev{}\sum_{i=1}^{|z_{t}|}|Z_{t}^{n,i}-z_{t}^{n,i}|1_{\cbr{\norm{z_{t}(i)}{}{}\geq t^{l}}}=e^{\alpha t}\ev{}|Z_{t}^{n,1}-z_{t}^{n,1}|1_{\cbr{\norm{z_{t}(1)}{}{}\geq t^{l}}}\\
\leq e^{\alpha t}\ev{}(Z_{t}^{n,1}-z_{t}^{n,1})^{2}\pr{\norm{z_{t}(1)}{}{}\geq t^{l}}\cleq e^{\alpha t}e^{-ct^{l}}\ev{}(Z_{t}^{n,1})^{2}\conv0\quad\text{as }t\conv+\infty.
\end{multline*}
To justify the above convergence it is enough to notice that by Fact
\ref{fac:backBoneFacts} we have $\ev{}(Z_{t}^{n,1})^{2}\cleq e^{ct}$,
for some $c>0$. By these and Fact \ref{fac:martingaleEquvalence}
the limit of $Z_{3}$ is the same as the one of 
\[
Z_{5}(t):=\rbr{\frac{\beta}{\alpha}e^{-\alpha t}|z_{t}|,e^{-(n\alpha/2)t}\sum_{i=1}^{\lfloor|\Lambda_{nt}|\rfloor}(1-V_{\infty}^{i}),e^{-(\alpha/2)t}\sum_{i=1}^{|z_{t}|}(Z_{t}^{n,i}-z_{t}^{n,i})1_{\cbr{\norm{z_{t}(i)}{}{}<t^{l}}}}.
\]
We denote now 
\[
Z_{6}(t):=\rbr{\frac{\beta}{\alpha}e^{-\alpha t}|z_{t}|,\lfloor|\Lambda_{nt}|\rfloor{}^{-1/2}\sum_{i=1}^{\lfloor|\Lambda_{nt}|\rfloor}(1-V_{\infty}^{i}),|z_{t}|^{-1/2}\sum_{i=1}^{|z_{t}|}(Z_{t}^{n,i}-z_{t}^{n,i})1_{\cbr{\norm{z_{t}(i)}{}{}<t^{l}}}}.
\]
Moreover, we consider the above quantity conditionally on the event
$\cbr{|\Lambda_{nt}|\geq t}$. We will denote the corresponding expectation
by $\ev{}_{\nu}^{t}$ and write

\begin{multline*}
\chi_{1}(\theta_{1},\theta_{2},\theta_{3};t):=\ev{}_{\nu}^{t}\exp\left\{ \ii\theta_{1}\frac{\beta}{\alpha}e^{-\alpha t}|z_{t}|+\ii\theta_{2}\lfloor|\Lambda_{nt}|\rfloor^{-1/2}\sum_{i=1}^{\lfloor|\Lambda_{nt}|\rfloor}(1-V_{\infty}^{i})\right.\\
\left.+\ii\theta_{3}|z_{t}|^{-1/2}\sum_{i=1}^{|z_{t}|}(Z_{t}^{n,i}-z_{t}^{n,i})1_{\cbr{\norm{z_{t}(i)}{}{}<t^{l}}}\right\} .
\end{multline*}
Let us denote by $h$ the characteristic function of $(1-V_{\infty}^{i})$.
Using the conditional expectation we obtain
\[
\chi_{1}(\theta_{1},\theta_{2},\theta_{3};t):=\ev{}_{\nu}^{t}\exp\cbr{\ii\theta_{1}\frac{\beta}{\alpha}e^{-\alpha t}|z_{t}|+\ii\theta_{3}|z_{t}|^{-1/2}\sum_{i=1}^{|z_{t}|}(Z_{t}^{n,i}-z_{t}^{n,i})1_{\cbr{\norm{z_{t}(i)}{}{}<t^{l}}}}h\rbr{\theta_{2}/\lfloor|\Lambda_{nt}|\rfloor^{-1/2}}^{\lfloor|\Lambda_{nt}|\rfloor}.
\]
By the central limit theorem we know that $h\rbr{\theta_{2}/\sqrt{n}}^{n}\conv e^{-\theta_{2}^{2}/(2\sigma_{V}^{2})}$,
where $\sigma_{V}=\frac{2\beta}{\alpha}$ (see \eqref{eq:varianceMaringaleLimit}).
Let us now denote 
\[
\chi_{2}(\theta_{1},\theta_{2},\theta_{3};t):=\ev{}_{\nu}^{t}\exp\cbr{\ii\theta_{1}\frac{\beta}{\alpha}e^{-\alpha t}|z_{t}|+\ii\theta_{3}|z(t)|^{-1/2}\sum_{i=1}^{|z_{t}|}(Z_{t}^{n,i}-z_{t}^{n,i})1_{\cbr{\norm{z_{t}(i)}{}{}<t^{l}}}}e^{-\theta_{2}^{2}/(2\sigma_{V}^{2})}.
\]
We check that 
\[
|\chi_{1}(\theta_{1},\theta_{2},\theta_{3};t)-\chi_{2}(\theta_{1},\theta_{2},\theta_{3};t)|\leq\ev{}_{\nu}^{t}\left|h\rbr{\theta_{2}\lfloor|X_{nt}|\rfloor^{-1/2}}^{\lfloor|X_{nt}|\rfloor}-e^{-\theta_{2}^{2}/(2\sigma_{V}^{2})}\right|\conv0,\quad\text{as }t\conv+\infty.
\]
Let us now notice that $\mathbb{P}_{\nu}(\epsilon\leq|X_{nt}|\leq t)\conv0$
for any $\epsilon>0$. Therefore $1_{\cbr{|X_{nt}|\geq t}}\conv1_{Ext^{c}}$.
We denote the expectation conditional on $Ext^{c}$ by $\tilde{\mathbb{E}}_{\nu}$.
Further we write

\[
\chi_{3}(\theta_{1},\theta_{2},\theta_{3};t):=\tilde{\mathbb{E}}_{\nu}\exp\cbr{\ii\theta_{1}\frac{\beta}{\alpha}e^{\alpha t}|z_{t}|+\ii\theta_{3}|z(t)|^{-1/2}\sum_{i=1}^{|z(t)|}(Z_{t}^{n,i}-z_{t}^{n,i})1_{\cbr{\norm{z_{t}(i)}{}{}<t^{l}}}}e^{-\theta_{2}^{2}/(2\sigma_{V}^{2})}.
\]
One easily checks that its limit is the same as the one of $\chi_{2}$. 

Let us now consider a sequence of $\cbr{a_{m}}_{m\geq0},\cbr{p_{m}}_{m\geq0}$
such that each $a_{m}\in\mathbb{N}$ and $p_{m}\in\R^{d\times a_{m}}$.
Moreover, we assume that $a_{m}\sim e^{\alpha m}$ and $\forall_{i\leq a_{m}}$
we have $\norm{p_{m}(i)}{}{}\leq m^{l}$. We denote 
\[
S_{m}:=a_{m}^{-1/2}\sum_{i=1}^{a_{m}}(\tilde{Z}{}_{m}^{n,i}-\tilde{z}_{m}^{n,i}),
\]
where $\tilde{Z}_{m}^{n,i}$ is defined as $Z_{t}^{n,i}$, but with
assumption that it starts off from position $p_{m}(i)$, and $\tilde{z}_{m}^{n,i}=\ev{}\tilde{Z}_{m}^{n,i}$
. 
\begin{multline*}
Var\rbr{a_{m}^{-1/2}\sum_{i=1}^{a_{m}}(\tilde{Z}_{m}^{n,i}-\tilde{z}_{m}^{n,i})}=a_{m}^{-1}\sum_{i=1}^{a_{m}}Var(\tilde{Z}_{m}^{n,i}-\tilde{z}_{m}^{n,i})\\
=a_{m}^{-1}\sum_{i=1}^{a_{m}}\rbr{\ev{}(\tilde{Z}_{m}^{n,i})^{2}-(\tilde{z}_{m}^{n,i})^{2}}=a_{m}^{-1}\sum_{i=1}^{a_{m}}\ev{}(\tilde{Z}_{m}^{n,i})^{2}-a_{m}^{-1}\sum_{i=1}^{a_{m}}(\tilde{z}_{m}^{n,i})^{2}=(*).
\end{multline*}
By Fact \ref{fac:backBoneFacts} and Fact \ref{fac:decay1} we know
that for some $k>0$ we have 
\begin{equation}
|\tilde{z}_{m}^{n,i}|\cleq e^{((n-1)\alpha/2)m}|\T{(n-1)m}\tilde{f}(p_{m}(i))|\cleq m^{2k}e^{(n-1)(\alpha/2-\mu)m}.\label{eq:helperEstimation}
\end{equation}
Therefore by assumption $\alpha<2\mu$ the second term of ({*}) disappears.
By Fact \ref{fac:backBoneFacts} we have
\[
\ev{}(Z_{m}^{n,i})^{2}=e^{-\alpha(n-1)m}V_{\tilde{f}}^{2}(p_{m}(i),(n-1)m).
\]
Using \eqref{eq:uniformConvergence} it is easy to check that $a_{m}^{-1}\sum_{i=1}^{a_{m}}\ev{}(\tilde{Z}_{m}^{n,i})^{2}\conv\sigma_{f}^{2}$.
Therefore, by the CLT, we know that
\[
S_{m}\conv^{d}\mathcal{N}(0,\sigma_{f}^{2}),
\]
once  we check the Lyapunov condition. To this end, we use \eqref{eq:helperEstimation}
and \eqref{eq:fourthMomentEstimate}, viz.,
\begin{multline*}
a_{m}^{-2}\sum_{i=1}^{a_{m}}\ev{}(\tilde{Z}_{m}^{n,i}-\tilde{z}_{m}^{n,i})^{4}\cleq a_{m}^{-2}\sum_{i=1}^{a_{m}}\sum_{k=0}^{4}|\tilde{z}_{m}^{n,i}|^{4-k}|\ev{}(Z_{m}^{n,i})^{k}|\cleq a_{m}^{-1}+a_{m}^{-2}\sum_{i=1}^{a_{m}}\sum_{k=2}^{4}\ev{}|Z_{m}^{n,i}|^{k}\\
\leq a_{m}^{-1}+a_{m}^{-2}\sum_{i=1}^{a_{m}}\sum_{k=2}^{4}\rbr{\ev{}|Z_{m}^{n,i}|^{4}}^{k/4}\cleq a_{m}^{-1}+a_{m}^{-2}\sum_{i=1}^{a_{m}}(1+\norm{p_{m}(i)}{}n)\cleq a_{m}^{-1}m^{ln}\conv0.
\end{multline*}
Applying the above CLT to $|z(t)|^{-1/2}\sum_{i=1}^{|z(t)|}(Z_{t}^{n,i}-z_{t}^{n,i})1_{\cbr{\norm{z_{t}(i)}{}{}<t}}$
(and performing the same reasoning as in case of $\chi_{1}$ and $\chi_{2}$)
one can show that the limit of $\chi_{3}$ is the same as the one
of 

\[
\chi_{4}(\theta_{1},\theta_{2},\theta_{3};t):=\tilde{\mathbb{E}}_{\nu}\exp\cbr{\ii\theta_{1}\frac{\beta}{\alpha}e^{-\alpha t}|z_{t}|}e^{-\theta_{3}^{2}/(2\sigma_{f}^{2})}e^{-\theta_{2}^{2}/(2\sigma_{V}^{2})},
\]
where $\sigma_{f}^{2}$ is given by \eqref{eq:sigmaDefinition}. Therefore,
by Fact \ref{fac:martingaleEquvalence}, we know that 
\[
\chi_{4}(\theta_{1},\theta_{2},\theta_{3};t)\conv\rbr{\tilde{\mathbb{E}}_{\nu}\exp\cbr{\ii\theta_{1}\frac{\beta}{\alpha}W_{\infty}}}e^{-\theta_{2}^{2}/(2\sigma_{V}^{2})}e^{-\theta_{3}^{2}/(2\sigma_{f}^{2})}=\rbr{\tilde{\mathbb{E}}_{\nu}\exp\cbr{\ii\theta_{1}V_{\infty}}}e^{-\theta_{2}^{2}/(2\sigma_{V}^{2})}e^{-\theta_{3}^{2}/(2\sigma_{f}^{2})}.
\]
Coming back step by step one sees that this proves that $Z_{6}(t)\conv^{d}(\tilde{V}{}_{\infty},G_{1},G_{2})$,
where the limit is as described in Theorem \ref{thm:smallBranching}.
This means that on $Ext^{c}$ we have convergence $Z_{5}(t)\conv^{d}(V_{\infty},V_{\infty}^{-1/2}G_{1},V_{\infty}^{-1/2}G_{2})$.
By standard arguments we also have $Z_{1}(t)\conv^{d}(V_{\infty},V_{\infty}^{-1/2}G_{1},V_{\infty}^{-1/2}G_{2})$
which is equivalent to the thesis of the theorem.

\section{Proof of Theorem \ref{thm:fastBranching}\label{sec:fastBranching}}

We first gather properties of $V_{f}^{k}$, extending the list in
Fact \ref{fac:backBoneFacts}, in the case when $\alpha>2\mu$. We
recall that $\tilde{f}(x)=f(x)-\ddp f{\eq}$. 
\begin{fact}
Let $\alpha>2\mu$ and $f\in\pspace(\R^{d})$. Then, there exists
$C,n>0$ such that 
\begin{equation}
e^{-2(\alpha-\mu)t}|V_{\tilde{f}}^{2}(x,t)|\leq C(1+\norm x{}n).\label{eq:secondMomentFast}
\end{equation}
\end{fact}
\begin{proof}
By \eqref{eq:secondMoment}, Fact \ref{fac:decay1} and using \eqref{eq:estimaiton77}
we obtain 
\begin{multline*}
|V_{\tilde{f}}^{2}(x,t)|\cleq\int_{0}^{t}\T{t-s}^{\alpha}\sbr{\rbr{\T s^{\alpha}\tilde{f}(\cdot)}^{2}}(x)\dd s+\int_{0}^{t}\T{t-s}^{\alpha}\sbr{\rbr{u_{\tilde{f}}^{*,1}(\cdot,s)}^{2}}(x)\dd s\\
+\int_{0}^{t}\T{t-s}^{\alpha}\sbr{u_{\tilde{f}}^{*,2}(\cdot,s)}(x)\dd s\cleq\intc te^{\alpha(t-s)}e^{2(\alpha-\mu)s}\T{t-s}\sbr{\rbr{1+\norm{\cdot}{}n}^{2}}(x)\dd s\\
+\int_{0}^{t}e^{\alpha(t-s)}e^{-2\alpha s}\T{t-s}\sbr{\rbr{1+\norm{\cdot}{}n}^{2}}(x)\dd s\\
+\int_{0}^{t}e^{\alpha(t-s)}e^{-2\alpha s}\T{t-s}\sbr{1+\norm{\cdot}{}{2n}}(x)\dd s\cleq e^{2(\alpha-\mu)t}(1+\norm x{}{2n}).
\end{multline*}

\end{proof}
We assume that $\Lambda$ has the backbone representation as described
in Section \ref{sub:Backbone-construction}. Although this is harmless
when we prove the weak convergence in \eqref{eq:factConv1} some additional
argument will be required in case of \eqref{eq:fastConv2}. This will
be explained at the end of the proof. Our first aim is to prove the
convergence of the spatial fluctuations. To this end we denote 
\[
Y_{1}(t):=e^{-(\alpha-\mu)t}\rbr{\ddp{\Lambda_{t}}f-|\Lambda_{t}|\ddp f{\eq}}=e^{-(\alpha-\mu)t}\ddp{\Lambda_{t}}{\tilde{f}},
\]
where as usual $\tilde{f}(x)=f(x)-\ddp f{\eq}$. We define 
\[
Y_{2}(s,t):=e^{-(\alpha-\mu)t}\sum_{i=1}^{|z_{t}|}e^{-(\alpha-\mu)s}\ddp{\Gamma_{s}^{i,t}}{\tilde{f}},
\]
By \eqref{eq:fundamentalDecomposition} and using the same argument
as in the previous proof (e.g. see \eqref{eq:tmp185}) one checks
that 

\[
|Y_{1}(t+s)-Y_{2}(s,t)|\conv0\quad\text{in probability,}
\]
if $s,t\conv+\infty$ and $s\geq t$ (this is a sufficient condition).
We write 
\[
Y_{2}(s,t)=e^{-(\alpha-\mu)t}\sum_{i=1}^{|z_{t}|}(Z_{s}^{i,t}-z_{s}^{i,t})+e^{-(\alpha-\mu)t}\sum_{i=1}^{|z_{t}|}z_{s}^{i,t}=:Y_{3}(s,t)+Y_{4}(s,t),
\]
where $Z_{s}^{i,t}=e^{-(\alpha-\mu)s}\ddp{\Gamma_{s}^{i,t}}{\tilde{f}}$
and $z_{s}^{i,t}=\ev{}\rbr{Z_{i}^{i,t}|z_{t}}$. One checks that
\begin{multline}
\ev{}Y_{3}(s,t)^{2}=e^{-2(\alpha-\mu)t}\ev{}\rbr{\sum_{i=1}^{|z_{t}|}\sum_{j=1}^{|z_{t}|}(Z_{s}^{i,t}-z_{s}^{i,t})(Z_{s}^{j,t}-z_{s}^{j,t})}\\
=e^{-2(\alpha-\mu)t}\ev{}\rbr{\ev{}\rbr{\left.\rbr{\sum_{i=1}^{|z_{t}|}\sum_{i=1}^{|z_{t}|}(Z_{s}^{i,t}-z_{s}^{i,t})(Z_{s}^{j,t}-z_{s}^{j,t})}\right|z_{t}}}\\
=e^{-2(\alpha-\mu)t}\ev{}\rbr{\sum_{i=1}^{|z_{t}|}\sum_{i=1}^{|z_{t}|}\ev{}\rbr{(Z_{s}^{i,t}-z_{s}^{i,t})(Z_{s}^{j,t}-z_{s}^{j,t})|z_{t}}}\\
=e^{-2(\alpha-\mu)t}\ev{}\rbr{\sum_{i=1}^{|z_{t}|}\ev{}\rbr{(Z_{s}^{i,t}-z_{s}^{i,t})^{2}|z_{t}}}.\label{eq:covarianceCalculations}
\end{multline}
By \eqref{eq:secondMomentFast}, Fact \ref{fac:backBoneFacts} and
Fact \ref{fac:decay1} we get $\ev{}\rbr{(Z_{s}^{i,t}-z_{s}^{i,t})^{2}|z_{t}}\cleq\norm{z_{t}(i)}{}n+1$,
for some $n\geq0$. And therefore by Fact \citep[(23)]{Adamczak:2011kx},
under assumption $\alpha>2\mu$, we obtain
\[
\ev{}Y_{3}(s,t)^{2}\cleq e^{-2(\alpha-\mu)t}\ev{}\rbr{\sum_{i=1}^{|z_{t}|}(\norm{z_{t}(i)}{}n+1)}\cleq e^{-2(\alpha-\mu)t}e^{\alpha t}\conv0,\quad\text{as }t\conv+\infty.
\]
By \eqref{eq:backBoneMoments} and \eqref{eq:mean} we have 
\begin{multline*}
z_{s}^{i,t}=\frac{\beta}{\alpha}\rbr{2e^{-\alpha s}\sinh(\alpha s)}\rbr{e^{\mu s}\T t\tilde{f}(z_{t}(i))}=\frac{\beta}{\alpha}\rbr{2e^{-\alpha s}\sinh(\alpha s)}\rbr{e^{\mu s}\T t\tilde{f}(z_{t}(i))-z_{t}(i)\circ\ddp{\grad f}{\eq}}\\
+\frac{\beta}{\alpha}\rbr{2e^{-\alpha s}\sinh(\alpha s)}\rbr{z_{t}(i)\circ\ddp{\grad f}{\eq}}=x_{s}^{i,t}+y_{s}^{i,t}.
\end{multline*}
This leads to
\[
Y_{4}(s,t)=Y_{5}(s,t)+Y_{6}(s,t):=e^{-(\alpha-\mu)t}\sum_{i=1}^{|z_{t}|}x_{s}^{i,t}+e^{-(\alpha-\mu)t}\sum_{i=1}^{|z_{t}|}y_{s}^{i,t}.
\]
By Fact \ref{fac:decay1} we have $|x_{s}^{i,t}|\leq\tilde{C}(1+\norm{z_{t}(i)}{}n)e^{-\mu s}$
for some $n\geq0$. Therefore
\[
|Y_{5}(s,t)|\leq e^{-(\alpha-\mu)t}e^{-\mu s}\sum_{i=1}^{|z_{t}|}(1+\norm{z_{t}(i)}{}n)=Ce^{\mu(t-s)}\conv0,
\]
whenever if $s,t\conv+\infty$ and $\frac{s}{t}\conv+\infty$. To
treat $Y_{6}(s,t)$ we recall \eqref{eq:secondMartingale}. Using
Fact \ref{fac:martingaleEquvalence} we obtain 
\begin{multline*}
Y_{6}(s,t):=\rbr{2\frac{\beta}{\alpha}e^{-\alpha s}\sinh(\alpha s)}e^{-(\alpha-\mu)t}\sum_{i=1}^{|z_{t}|}\rbr{z_{t}(i)\circ\ddp{\grad f}{\eq}}\\
=\rbr{2\frac{\beta}{\alpha}e^{-\alpha s}\sinh(\alpha s)}\rbr{\ddp{\grad f}{\eq}\circ e^{-(\alpha-\mu)t}\sum_{i=1}^{|z_{t}|}z_{t}(i)}\\
\conv\frac{\beta}{\alpha}\rbr{\ddp{\grad f}{\eq}\circ I_{\infty}}=\ddp{\grad f}{\eq}\circ H_{\infty},\quad\text{a.s.}
\end{multline*}
as $s,t\conv+\infty$. Reviewing the above steps one checks easily
that 
\begin{equation}
Y_{1}(t)\conv\ddp{\grad f}{\eq}\circ H_{\infty},\quad\text{in probability,}\label{eq:tmpConv}
\end{equation}
as $t\conv+\infty$. This was proved for a special (backbone) realization
of the superprocess. Let us fix some other realization $\hat{X}$,
denote $\hat{Y}_{1}(t):=e^{-(\alpha-\mu)t}\ddp{\hat{X}_{t}}{\tilde{f}}$
and by $\hat{H}_{\infty}$ the limit of corresponding martingale (see
\eqref{eq:secondMartingale}). From \eqref{eq:tmpConv} we can conclude
only that $\hat{Y}_{1}(t)\conv\ddp{\grad f}{\eq}\circ H_{\infty}$
in distribution. However, from \eqref{eq:tmpConv} we know that $\rbr{Y_{1}(t)-\ddp{\grad f}{\eq}\circ H_{\infty}}\conv0$
in probability. From this we conclude that $\rbr{\hat{Y}_{1}(t)-\ddp{\grad f}{\eq}\circ\hat{H}_{\infty}}\conv0$
in distribution and hence also in probability. In this way \eqref{eq:tmpConv}
holds for any realization of the superprocess. 

To conclude the proof we notice that to obtain the joint convergence
in \eqref{eq:factConv1} one needs to use the same methods as in the
proof of Theorem \ref{thm:smallBranching}. 

In this section we also present 
\begin{proof}
(of Fact \ref{fac:Hmartingale})The fact that $H$ is a martingale
follows directly from the Markov property of $X$, \eqref{eq:superpocesMoment1}
and \eqref{eq:semigroupFirstEignenvalue}, viz.,
\[
\ev{}_{\nu}\ddp{X_{t}}{id}=\int_{\R^{d}}u_{id}^{1}(x,t)\nu(\dd x)=e^{\alpha t}\int_{\R^{d}}\T tid(x)\nu(\dd x)=e^{(\alpha-\mu)t}\int_{\R^{d}}x\nu(\dd x),
\]
where $id(x)=x$. By standard Laplace transform arguments one easily
checks that $H$ is $L^{2}$-bounded if only $e^{-2(\alpha-\mu)t}|u_{id}^{2}|\leq c(x)$,
for any $t\geq0$ and some function $c(x)$. We recall \eqref{eq:superprocesMoment2},
Fact \ref{fac:decay1} and write 
\begin{multline*}
e^{-2(\alpha-\mu)t}|u_{id}^{2}(x,t)|\leq e^{-2(\alpha-\mu)t}\int_{0}^{t}\T{t-s}^{\alpha}\sbr{(\T s^{\alpha}id(\cdot))^{2}}(x)\dd s=e^{-2(\alpha-\mu)t}\int_{0}^{t}\T{t-s}^{\alpha}\sbr{e^{2(\alpha-\mu)s}id(\cdot)^{2}}(x)\dd s\\
=e^{-(\alpha-2\mu)t}\int_{0}^{t}e^{(\alpha-2\mu)s}\T{t-s}\sbr{id(\cdot)^{2}}(x)\dd s\cleq(1+\norm x{}n).
\end{multline*}

\end{proof}

\section{Proof of Theorem \ref{thm:criticalBranching}\label{sec:criticalBranching}}

We first gather properties of $V_{f}^{k}$, extending the list in
Fact \ref{fac:backBoneFacts}, in the case when $\alpha=2\mu$. We
recall that $\tilde{f}(x)=f(x)-\ddp f{\eq}$. 
\begin{fact}
\label{fac:criticalBranchingMoments}Let $\alpha=2\mu$ and $f\in\pspace(\R^{d})$.
Then for any $x\in\R^{d}$ there is 
\begin{equation}
t^{-1/2}e^{-(\alpha/2)t}V_{\tilde{f}}^{1}(x,t)\conv0,\quad\text{as }t\conv+\infty,\label{eq:criticalBranchingMoment1}
\end{equation}
\textup{
\begin{equation}
t^{-1}e^{-\alpha t}V_{\tilde{f}}^{2}(x,t)\conv\sigma_{f}^{2},\quad\text{as }t\conv+\infty,\label{eq:criticalBranchingMoment2}
\end{equation}
where $\sigma_{f}^{2}$ is given by }\eqref{eq:criticalSigma}\textup{.
We have $\sigma_{f}^{2}<+\infty$. Moreover, there exists $l>0$ such
that 
\begin{equation}
\sup_{\norm x{}{}\leq t^{l}}|t^{-1}e^{-\alpha t}V_{\tilde{f}}^{2}(x,t)-\sigma_{f}^{2}|\conv0,\quad\text{as }t\conv+\infty.\label{eq:criticalBranchingUniformConvergence}
\end{equation}
Finally }there exists $C,n>0$ such that 
\begin{equation}
|V_{\tilde{f}}^{4}(x,t)|\leq Ct^{2}e^{2\alpha t}(\norm x{}{2n}+1).\label{eq:criticalBranchingFourthMoment}
\end{equation}
\end{fact}
\begin{proof}
\eqref{eq:criticalBranchingMoment1} follows by \eqref{eq:mean} and
Fact \ref{fac:decay1}. Indeed
\[
t^{-1/2}e^{-(\alpha/2)t}V_{\tilde{f}}^{1}(x,t)=e^{-(\alpha/2)t}t^{-1/2}e^{\alpha t}|\T t\tilde{f}(x)|\leq t^{-1/2}e^{(\alpha/2)t}e^{-\mu t}(1+\norm x{}{2n})\conv0.
\]
Using \eqref{eq:secondMoment} we have 
\begin{multline*}
t^{-1}e^{-\alpha t}V_{\tilde{f}}^{2}(x,t)=e^{-\alpha t}t^{-1}\frac{\beta}{\alpha}\int_{0}^{t}\T{t-s}^{\alpha}\sbr{2\beta\rbr{\T s^{\alpha}\tilde{f}(\cdot)}^{2}-2\beta\rbr{\T s^{-\alpha}\tilde{f}(\cdot)}^{2}-2\alpha u_{\tilde{f}}^{*,2}(\cdot,s)}(x)\dd s\\
=t^{-1}\frac{\beta}{\alpha}\int_{0}^{t}e^{-\alpha s}\T{t-s}\sbr{2\beta\rbr{\T s^{\alpha}\tilde{f}(\cdot)}^{2}-2\beta\rbr{\T s^{-\alpha}\tilde{f}(\cdot)}^{2}-2\alpha u_{\tilde{f}}^{*,2}(\cdot,s)}(x)\dd s\\
=t^{-1}\frac{\beta}{\alpha}\int_{0}^{t}\T{t-s}\sbr{2\beta\rbr{\T s^{\mu}\tilde{f}(\cdot)}^{2}}(x)\dd s+t^{-1}\frac{\beta}{\alpha}\int_{0}^{t}e^{-\alpha s}\T{t-s}\sbr{-2\beta\rbr{\T s^{-\alpha}\tilde{f}(\cdot)}^{2}-2\alpha u_{\tilde{f}}^{*,2}(\cdot,s)}(x)\dd s.
\end{multline*}
By Fact \ref{fac:misc1} it is easy to check that the second term
converges to $0$. Using Fact \ref{fac:decay1} (and \eqref{eq:decaySpeed}
in particular) we have
\begin{multline*}
\left|t^{-1}\frac{\beta}{\alpha}\int_{0}^{t}\T{t-s}\sbr{2\beta\rbr{\T s^{\mu}\tilde{f}(\cdot)}^{2}}(x)\dd s-t^{-1}\frac{\beta}{\alpha}\int_{0}^{t}\T{t-s}\sbr{2\beta\rbr{\cdot\circ\ddp{\grad f}{\eq}}^{2}}(x)\dd s\right|\\
\cleq t^{-1}\int_{0}^{t}\T{t-s}\sbr{\left|\T s^{\mu}\tilde{f}(\cdot)-\rbr{\cdot\circ\ddp{\grad f}{\eq}}\right|\left|\T s^{\mu}\tilde{f}(\cdot)+\rbr{\cdot\circ\ddp{\grad f}{\eq}}\right|}(x)\dd s\\
\cleq t^{-1}\int_{0}^{t}\T{t-s}\sbr{e^{-\mu s}(1+\norm{\cdot}{}n)(1+\norm{\cdot}{}n)}(x)\dd s\cleq t^{-1}(1+\norm x{}n)\conv0.
\end{multline*}
Using standard integral arguments one easily shows that $t^{-1}\frac{\beta}{\alpha}\int_{0}^{t}\T{t-s}\sbr{2\beta\rbr{\cdot\circ\ddp{\grad f}{\eq}}^{2}}(x)\dd s\conv\sigma_{f}^{2}$
and that $\sigma_{f}^{2}<+\infty$. 

Now in order to prove \eqref{eq:criticalBranchingUniformConvergence}
it is enough to show 
\[
\sup_{\norm x{}{}\leq t^{l}}t^{-1}e^{-\alpha t}|V_{\tilde{f}}^{2}(x,t)-V_{\tilde{f}}^{2}(0,t)|\conv0,\quad\text{as }t\conv+\infty.
\]
We will use the same decomposition as before and concentrate only
on the first term (leaving other, simpler ones, to the reader), that
is, we will show that 
\[
\sup_{\norm x{}{}\leq t^{l}}t^{-1}\left|\int_{0}^{t}e^{-\alpha s}\T{t-s}\sbr{2\beta\rbr{\T s^{\alpha}\tilde{f}(\cdot)}^{2}}(x)\dd s-\int_{0}^{t}e^{-\alpha s}\T{t-s}\sbr{2\beta\rbr{\T s^{\alpha}\tilde{f}(\cdot)}^{2}}(0)\dd s\right|\conv0.
\]
By Fact \ref{fac:decay1} 
\begin{multline*}
\sup_{\norm x{}{}\leq t^{l}}t^{-1}\left|\int_{t-t^{-1/2}}^{t}e^{-\alpha s}\T{t-s}\sbr{2\beta\rbr{\T s^{\alpha}\tilde{f}(\cdot)}^{2}}(x)\dd s\right|\cleq\\
\sup_{\norm x{}{}\leq t^{l}}t^{-1}\left|\int_{t-t^{-1/2}}^{t}e^{(\alpha-2\mu)s}\T{t-s}\sbr{\rbr{(1+\norm{\cdot}{}n)}^{2}}(x)\dd s\right|\cleq t^{-1}t^{1/2}t^{nl}\conv0,
\end{multline*}
if only $l$ is small enough. Further we proceed as in \eqref{eq:kurnik}
putting $t-t^{1/2}$ instead of $t/2$.

To prove \eqref{eq:criticalBranchingFourthMoment} we will follow
the same route as in the proof of Fact \ref{fac:slowBranchingMoments};
let us recall \eqref{eq:generalEquation-1}. Analogously we omit terms
$u_{f}^{*,k}$. For $k=3$ we calculate 
\begin{multline*}
\int_{0}^{t}\T{t-s}^{\alpha}\sbr{\left|V_{\tilde{f}}^{2}(\cdot,s)V_{\tilde{f}}^{1}(\cdot,s)\right|}(x)\dd s\leq e^{\alpha t}\intc t\T{t-s}e^{-\alpha s}\sbr{\left|V_{\tilde{f}}^{2}(\cdot,s)e^{(\alpha-\mu)s}(1+\norm{\cdot}{}n)\right|}(x)\dd s\\
\leq e^{\alpha t}\intc t\T{t-s}e^{-\mu s}\sbr{\left|se^{\alpha s}(1+\norm{\cdot}{}{2n})(1+\norm{\cdot}{}n)\right|}(x)\dd s\cleq e^{(3\alpha/2)t}t(1+\norm{\cdot}{}{3n}).
\end{multline*}
We proceed to $k=4$. Again following the framework of Fact \ref{fac:slowBranchingMoments}
we write 
\begin{multline*}
\int_{0}^{t}\T{t-s}^{\alpha}\sbr{\left|V_{\tilde{f}}^{3}(\cdot,s)V_{\tilde{f}}^{1}(\cdot,s)\right|}(x)\dd s+\int_{0}^{t}\T{t-s}^{\alpha}\sbr{\left|V_{\tilde{f}}^{2}(\cdot,s)\right|^{2}}(x)\dd s\\
\leq\int_{0}^{t}\T{t-s}^{\alpha}\sbr{\left|se^{(3\alpha/2)s}(1+\norm{\cdot}{}{3n})e^{(\alpha-\mu)s}(1+\norm{\cdot}{}{3n})\right|}(x)\dd s\\
+\int_{0}^{t}\T{t-s}^{\alpha}\sbr{\left|se^{\alpha s}(1+\norm{\cdot}{}{2n})\right|^{2}}(x)\dd s\cleq t^{2}e^{2\alpha t}(1+\norm{\cdot}{}{4n}).
\end{multline*}

\end{proof}
Finally we are able to prove Theorem \ref{thm:criticalBranching}.
We start with the following random vector 
\[
Z_{1}(t):=\rbr{e^{-\alpha t}|\Lambda_{t}|,e^{-(\alpha/2)t}(|\Lambda_{t}|-e^{\alpha t}V_{\infty}),t^{-1/2}e^{-(\alpha/2)t}\ddp{\Lambda_{t}}{\tilde{f}}}.
\]
Let $n\in\mathbb{N}$ to be fixed later; the limit of $Z_{1}(nt)$
is obviously the same as the on of $Z_{1}(t)$. Using Fact \ref{fac:totalMassMartingale}
we obtain that $V_{nt}-V_{t}\conv0$ a.s. Therefore the limit of $Z_{1}(t)$
coincides with the one of 
\[
Z_{2}^{n}(t):=\rbr{e^{-\alpha t}|\Lambda_{t}|,e^{-(n\alpha/2)t}(|\Lambda_{nt}|-e^{n\alpha t}V_{\infty}),(nt)^{-1/2}e^{-(n\alpha/2)t}\ddp{\Lambda_{nt}}{\tilde{f}}}.
\]
We denote $Z_{t}^{n,i}:=\rbr{\frac{n-1}{n}}^{1/2}((n-1)t)^{-1/2}e^{-((n-1)\alpha/2)t}\ddp{\Gamma_{(n-1)t}^{i,t}}{\tilde{f}}$
and $z_{t}^{n,i}:=\ev{}\rbr{Z_{t}^{n,i}|z_{t}}$. We keep the notation
consistent to the one in the proof of Theorem \ref{thm:smallBranching}
but we add superscript $^{n}$ as this parameter will vary. Using
argument of the aforementioned proof we have that the limit of $Z_{2}^{n}$
is the same as the one of
\[
Z_{3}^{n}(t):=\rbr{e^{-\alpha t}|\Lambda_{t}|,e^{(n\alpha/2)t}\sum_{i=1}^{\lfloor|\Lambda_{nt}|\rfloor}(1-V_{\infty}^{i}),e^{-(\alpha/2)t}\sum_{i=1}^{|z_{t}|}(Z_{t}^{n,i}-z_{t}^{n,i})+H_{t}^{n}},
\]
where $H_{t}^{n}:=e^{-(\alpha/2)t}\sum_{i=1}^{|z_{t}|}z_{t}^{n,i}$.
Applying \eqref{eq:mean} to $z_{t}^{n,i}$ we obtain 
\[
H_{t}^{n}=2\frac{\beta}{\alpha}(nt)^{-1/2}e^{-((n\alpha/2)t}\sinh((n-1)\alpha t)\sum_{i=1}^{|z_{t}|}\T{(n-1)}\tilde{f}(z_{t}(i)).
\]
Using \eqref{eq:secondMoment} we get 
\[
\ev{}\rbr{H_{t}^{n}}^{2}\cleq(nt)^{-1}e^{\alpha(n-2)t}V_{g}^{2}(x,t),
\]
where $g(x):=\T{(n-1)t}\tilde{f}(x).$ By \eqref{eq:superpocesMoment1}
we have $|u_{f}^{*,1}(x,s)|=|e^{-\alpha s}\T{(n-1)t+s}\tilde{f}(x)|\leq e^{-\alpha s}e^{-\mu((n-1)t+s)}(1+\norm x{}{2n})\leq e^{-\alpha s}e^{-\mu(n-1)t}(1+\norm x{}{2n})$.
As usual $u_{f}^{*,1}$ can be skipped. Further
\begin{multline*}
\int_{0}^{t}\T{t-s}^{\alpha}\sbr{\rbr{\T s^{\alpha}g(\cdot)}^{2}}(x)\dd s=e^{\alpha t}\int_{0}^{t}e^{\alpha s}\T{t-s}\sbr{\rbr{\T{s+(n-1)t}\tilde{f}(\cdot)}^{2}}(x)\dd s\\
\cleq e^{\alpha t}\int_{0}^{t}e^{\alpha s}\T{t-s}\sbr{\rbr{e^{-\mu(s+(n-1)t)}(1+\norm{\cdot}{}{2n}}^{2}}(x)\dd s\\
=e^{-\alpha(n-2)t}\intc t\T{t-s}\sbr{\rbr{(1+\norm{\cdot}{}{2n}}^{2}}(x)\dd s\cleq te^{-\alpha(n-2)t}.\\
\end{multline*}
Therefore there exists a constant $C>0$ such that 
\begin{equation}
\ev{}\rbr{H_{t}^{n}}^{2}\leq\frac{C}{n}.\label{eq:hnt}
\end{equation}
In fact also the reverse inequality holds, hence $H^{n}$ is not negligible
(unlike in the proof of Theorem \ref{thm:smallBranching}). We skip
it for a moment and denote 
\[
Z_{4}^{n}(t):=\rbr{e^{-\alpha t}|\Lambda_{t}|,e^{(n\alpha/2)t}\sum_{i=1}^{\lfloor|\Lambda_{nt}|\rfloor}(1-V_{\infty}^{i}),e^{-(\alpha/2)t}\sum_{i=1}^{|z_{t}|}(Z_{t}^{n,i}-z_{t}^{n,i})}.
\]
Further, we consider 
\begin{multline*}
I(t):=\ev{}\rbr{e^{-(\alpha/2)t}\sum_{i=1}^{|z_{t}|}(Z_{t}^{n,i}-z_{t}^{n,i})1_{\cbr{\norm{z_{t}(i)}{}{}\geq t^{l}}}}^{2}\\
=e^{-\alpha t}\ev{}\rbr{\sum_{i=1}^{|z_{t}|}\sum_{j=1}^{|z_{t}|}\ev{}\rbr{(Z_{t}^{n,i}-z_{t}^{n,i})(Z_{t}^{n,j}-z_{t}^{n,j})1_{\cbr{\norm{z_{t}(i)}{}{}\geq t^{l}}}1_{\cbr{\norm{z_{t}(j)}{}{}\geq t^{l}}}|z_{t}}}\\
=e^{-\alpha t}\ev{}\rbr{\sum_{i=1}^{|z_{t}|}\ev{}\rbr{(Z_{t}^{n,i}-z_{t}^{n,i})^{2}|z_{t}}1_{\cbr{\norm{z_{t}(i)}{}{}\geq t^{l}}}}=e^{-\alpha t}\ev{}\rbr{\sum_{i=1}^{|z_{t}|}\ev{}\rbr{(Z_{t}^{n,i})^{2}|z_{t}}1_{\cbr{\norm{z_{t}(i)}{}{}\geq t^{l}}}}\\
+e^{-\alpha t}\ev{}\rbr{\sum_{i=1}^{|z_{t}|}(z_{t}^{n,i})^{2}1_{\cbr{\norm{z_{t}(i)}{}{}\geq t^{l}}}}=:I_{1}(t)+I_{2}(t).
\end{multline*}
Our aim is to prove that both $I_{1}(t)\conv0,\: I_{2}(t)\conv0$
and hence $I(t)$ is negligible. Let us recall \eqref{eq:mean}, Fact
\ref{fac:decay1} and write 
\begin{multline*}
I_{2}(t)\cleq t^{-1}e^{-n\alpha t}\ev{}\rbr{\sum_{i=1}^{|z_{t}|}(e^{(n-1)\alpha t}\T{(n-1)t}\tilde{f}(z_{t}(i))^{2}1_{\cbr{\norm{z_{t}(i)}{}{}\geq t^{l}}}}\\
\cleq t^{-1}e^{(n-2)\alpha t}\ev{}\rbr{\sum_{i=1}^{|z_{t}|}(e^{-(n-1)\mu t}(1+\norm{z_{t}(i)}{}k))^{2}1_{\cbr{\norm{z_{t}(i)}{}{}\geq t^{l}}}}\\
\cleq t{}^{-1}e^{-\alpha t}\ev{}\rbr{\sum_{i=1}^{|z_{t}|}(1+\norm{z_{t}(i)}{}{2k})1_{\cbr{\norm{z_{t}(i)}{}{}\geq t^{l}}}}\\
=t{}^{-1}\ev{}\rbr{(1+\norm{z_{t}(1)}{}{2k})1_{\cbr{\norm{z_{t}(1)}{}{}\geq t^{l}}}}\conv0.
\end{multline*}
To estimate $I_{1}$, we recall \eqref{eq:secondMoment} and write
\[
\ev{}\rbr{\ddp{\Gamma_{(n-1)t}^{i,t}}{\tilde{f}}^{2}|z_{t}}=\frac{\beta}{\alpha}\int_{0}^{(n-1)t}\T{(n-1)t-s}^{\alpha}\sbr{2\beta\rbr{\T s^{\alpha}\tilde{f}(\cdot)}^{2}-2\beta\rbr{u_{\tilde{f}}^{*,1}(\cdot,s)}^{2}-2\alpha u_{\tilde{f}}^{*,2}(\cdot,s)}(z_{t}(i))\dd s.
\]
As usual it is enough to focus on the first term 
\begin{multline*}
\frac{\beta}{\alpha}\int_{0}^{(n-1)t}\T{(n-1)t-s}^{\alpha}\sbr{2\beta\rbr{\T s^{\alpha}\tilde{f}(\cdot)}^{2}}(z_{t}(i))\dd s\\
\cleq e^{(n-1)\alpha t}\int_{0}^{(n-1)t}\T{(n-1)t-s}e^{-\alpha s}\sbr{\rbr{e^{\alpha s}e^{-\mu s}(1+\norm{\cdot}{}k)}^{2}}(z_{t}(i))\dd s\\
e^{(n-1)\alpha t}\int_{0}^{(n-1)t}\T{(n-1)t-s}\sbr{\rbr{(1+\norm{\cdot}{}k)}^{2}}(z_{t}(i))\dd s\cleq e^{(n-1)\alpha t}(n-1)t(1+\norm{z_{t}(i)}{}{2k}).
\end{multline*}
Therefore 
\[
\ev{}\rbr{(Z_{t}^{n,i})^{2}|z_{t}}=(nt)^{-1}e^{((n-1)\alpha)t}\ev{}\rbr{\ddp{\Gamma_{(n-1)t}^{i,t}}{\tilde{f}}^{2}|z_{t}}\cleq1+\norm{z_{t}(i)}{}{2k}.
\]
Using the fact that $z_{t}(1)$ is Gaussian with bounded variance
we conclude 
\begin{multline*}
I_{1}(t)\cleq e^{-\alpha t}\ev{}\sum_{i=1}^{|z_{t}|}(1+\norm{z_{t}(i)}{}{2k})1_{\cbr{\norm{z_{t}(i)}{}{}\geq t^{l}}}=\ev{}(1+\norm{z_{t}(1)}{}{2k})1_{\cbr{\norm{z_{t}(1)}{}{}\geq t^{l}}}\\
\cleq\rbr{\ev{}(1+\norm{z_{t}(1)}{}{4k})}^{1/2}\pr{\norm{z_{t}(1)}{}{}\geq t^{l}}^{1/2}\conv0.
\end{multline*}
Therefore the limit of $Z_{4}^{n}$ is the same as the one of 
\[
Z_{5}^{n}(t):=\rbr{\frac{\beta}{\alpha}e^{-\alpha t}|z_{t}|,\lfloor|X_{nt}|\rfloor{}^{-1/2}\sum_{i=1}^{\lfloor|X_{nt}|\rfloor}(1-V_{\infty}^{i}),|z_{t}|^{-1/2}\sum_{i=1}^{|z_{t}|}(Z_{t}^{n,i}-z_{t}^{n,i})1_{\cbr{\norm{z_{t}(i)}{}{}<t^{l}}}}.
\]
Further we proceed along the lines of the proof of Theorem \ref{thm:smallBranching}
using Fact \ref{fac:criticalBranchingMoments} instead of Fact \ref{fac:slowBranchingMoments}.
There are some minor changes. We describe two of them which are not
completely obvious. \eqref{eq:helperEstimation} now writes as
\[
|\tilde{z}_{m}^{n,i}|\cleq m{}^{-1/2}e^{((n-1)\alpha/2)m}|\T{(n-1)m}\tilde{f}(p_{i}(m))|\cleq m^{-1/2}m^{2kl}e^{(n-1)(\alpha/2-\mu)m}=m^{2kl-1/2}.
\]
Decreasing $l$ if necessary, let us note that any other part of the
proof works for any {}``small'' $l$, we make $2kl-1/2<0$ and therefore
get the convergence to $0$. The second change is that 
\[
a_{m}^{-1}\sum_{i=1}^{a_{m}}\ev{}\rbr{\tilde{Z}_{m}^{n,i}}^{2}\conv\rbr{\frac{n-1}{n}}\sigma_{f}^{2}.
\]
The additional term $\rbr{\frac{n-1}{n}}^{1/2}$ stems from the corresponding
term in definition of $Z_{t}^{n,i}$. Finally we arrive at 
\begin{equation}
Z_{5}^{n}(t)\conv^{d}\rbr{V_{\infty},\sqrt{V_{\infty}}G_{1},\rbr{\frac{n-1}{n}}^{1/2}\sqrt{V_{\infty}}G_{2}},\quad\text{as }t\conv+\infty,\label{eq:limit}
\end{equation}
conditionally on $Ext^{c}$, where $G_{1},G_{2}$ are the same as
in Theorem \ref{thm:fastBranching}.

To finish the proof we need some topological notions. Let $\mu_{1},\mu_{2}$
be two probability measures on $\R$, and $\text{Lip}(1)$ be the
space of continuous functions $\R\mapsto[-1,1]$ with the Lipschitz
constant smaller or equal to $1$. We define
\[
m(\mu_{1},\mu_{2}):=\sup_{g\in\text{Lip}(1)}|\ddp g{\mu_{1}}-\ddp g{\mu_{2}}|.
\]
It is well known that $m$ is a distance equivalent to weak convergence.
One easily checks that when $\mu_{1},\mu_{2}$ correspond to two random
variables $X_{1},X_{2}$ on the same probability space then we have
\[
m(\mu_{1},\mu_{2})\leq\norm{X_{1}-X_{2}}1{}\leq\sqrt{\norm{X_{1}-X_{2}}2{}}.
\]
We denote the law of of the triple the limit \eqref{eq:limit} by
$\mathcal{L}_{n}$ and the law of $(W,\sqrt{W}G_{1},\sqrt{W}G_{2})$
by $\mathcal{L}_{\infty}$. Let us fix $\epsilon>0$. We choose $n$
such that $\sqrt{C/n}\le\varepsilon$, where $C$ is the same as in
\eqref{eq:hnt}. Hence, for any $t$ we have $m(Z_{4}^{n}(t),Z_{3}^{n}(t))\leq\epsilon$. 

By the fact that $Z_{5}^{n}$ and $Z_{4}^{n}$ have the same limit
there exist $T_{1}^{n}>0$ such that for any $t>T_{1}^{n}$ we have
$m(Z_{4}^{n}(t),\mathcal{L}_{n})\leq\epsilon$. Analogously $Z_{1}$
has the same limit as $Z_{3}^{n}$ therefore there exists $T_{2}^{n}$
such that for any $t>T_{2}^{n}$ we have $m(Z_{1}(t),Z_{3}^{n}(t))\leq\epsilon$.
Applying the triangle inequality we obtain 
\[
m(Z_{1}(t),\mathcal{L}_{\infty})\leq3\epsilon,
\]
if only $t\geq T_{1}^{n}\vee T_{2}^{n}$. The proof is concluded since
$\epsilon$ can be taken arbitrarily small.

\subsection*{Appendix}
\begin{proof}
(of Fact \eqref{fac:totalMassMartingale}) Using \eqref{eq:logLaplace}
and \eqref{eq:integralEq} one easily checks that 
\[
\mathbb{E}_{\nu}(e^{-\theta|X_{t}|})=\exp\rbr{-|\nu|v_{\theta}(t)},
\]
where 
\[
v_{\theta}'(t)=\alpha v_{\theta}(t)-\beta v_{\theta}(t)^{2},\quad v_{\theta}(0)=\theta.
\]
This equation can be solved explicitly, viz.
\begin{equation}
v_{\theta}(t)=\frac{e^{t\alpha}\alpha}{C_{\theta}+e^{t\alpha}\beta},\quad C_{\theta}:=\frac{\alpha-\beta\theta}{\theta}.\label{eq:totalMassSolution}
\end{equation}
By direct calculations we obtain $\ev{}_{\nu}|X_{t}|=e^{\alpha t}|\nu|$.
Hence, using the Markov property of $X$ we conclude that $V_{t}$
is a martingale.

One checks that $\lim_{t\conv+\infty}v_{\theta}(t)=\alpha/\beta$.
As the limit does not depend on $\theta$ we conclude that $|X_{t}|$
either diverges up to infinity or converges to $0$. By definition
the latter event is the same as $Ext$ (see \eqref{eq:extenquished})
thus we conclude that 
\[
\mathbb{P}_{\nu}(Ext)=\exp\cbr{-|\nu|\frac{\alpha}{\beta}}.
\]
Let $L(\theta;t):=-\log\mathbb{E}_{\nu}(e^{-\theta e^{-\alpha t}|X_{t}|})/|\nu|.$
We calculate
\[
L(\theta;t)=\frac{e^{t\alpha}\alpha}{\frac{\alpha-\beta\theta e^{-\alpha t}}{\theta e^{-\alpha t}}+e^{t\alpha}\beta}=\frac{e^{t\alpha}\alpha}{\alpha\theta^{-1}e^{\alpha t}-\beta+e^{t\alpha}\beta}\conv\frac{\alpha}{\alpha\theta^{-1}+\beta},\quad\text{as }t\conv+\infty.
\]
Now we check that $\frac{\alpha}{\alpha\theta^{-1}+\beta}\conv\alpha/\beta$
as $\theta\conv+\infty$, hence $\cbr{V_{\infty}=0}=Ext$. One may
verify representation \eqref{eq:LimitRep} by direct calculations
(although it is easier guessed from the backbone construction). In
order to prove $L^{2}$ convergence one needs to apply \eqref{eq:superprocesMoment2}
with $f=1$. 
\end{proof}

\begin{proof}
(of Fact \ref{fac:martingaleEquvalence}) The fact that $W$ is a
martingale is a well-known fact from the Galton-Watson theory (see
e.g.\citep{Athreya:2004xr}). The properties of $I$ are proved in
\citep{Adamczak:2011kx} in Section 3.4.3 (where it is under name
$H$). Having $a.s$ convergence of $H$ and $I$ in order to prove
\eqref{eq:martingaleEquiv2} one needs to show that 
\[
H_{nt}-\frac{\beta}{\alpha}I_{t}\conv0,
\]
 in probability. We denote $id(x)=x$ and fix some $n>2$. Using \eqref{eq:fundamentalDecomposition}
we obtain
\begin{multline*}
H_{nt}-\frac{\beta}{\alpha}I_{t}\\=e^{-n(\alpha-\mu)t}\ddp{X_{nt}^{0}}{id}+e^{-n(\alpha-\mu)t}\ddp{D_{(n-1)t}^{t}}{id}+e^{-(\alpha-\mu)t}\sum_{i=1}^{|z_{t}|}e^{-n(\alpha-\mu)t}\ddp{\Gamma_{(n-1)t}^{i,t}}{id}-e^{-(\alpha-\mu)t}\sum_{i=1}^{|z_{t}|}\frac{\beta}{\alpha}z_{t}(i)\\
=e^{-n(\alpha-\mu)t}\ddp{X_{nt}^{0}}{id}+e^{-n(\alpha-\mu)t}\ddp{D_{(n-1)t}^{t}}{id}+e^{-(\alpha-\mu)t}\sum_{i=1}^{|z_{t}|}\rbr{Z^{i,t}-z^{i,t}}+e^{-(\alpha-\mu)t}\sum_{i=1}^{|z_{t}|}\rbr{z^{i,t}-z_{t}(i)}\\
=:I_{1}(t)+I_{2}(t)+I_{3}(t)+I_{4}(n),
\end{multline*}
where $ $$Z^{i,t}:=e^{-(n-1)(\alpha-\mu)t}\ddp{\Gamma_{(n-1)t}^{i,t}}{id}$
and $z^{i,t}:=\ev{}(Z^{i,t}|z_{t})$. By \eqref{eq:tmp17} we have
\[
\ev{}_{\nu}|I_{2}(t)|\leq e^{-n(\alpha-\mu)t}\ev{}\ddp{D_{(n-1)t}^{t}}{|id|}\cleq e^{-n(\alpha-\mu)t}e^{(2-n)t}\conv0.
\]
Analogously using \eqref{eq:superpocesMoment1} (with $\psi^{*}$
as explained in Section \ref{sub:Backbone-construction})one can show
that $\ev{}_{\nu}^{*}|I_{1}(t)|\conv0$, at $t\conv+\infty$. In order
to treat $I_{3}$ we perform the same calculations as for \eqref{eq:covarianceCalculations}
and obtain 
\[
\ev{}I_{2}(t)^{2}=e^{-2(\alpha-\mu)t}\ev{}\rbr{\sum_{i=1}^{|z_{t}|}\ev{}\rbr{(Z^{i,t}-z^{i,t})^{2}|z_{t}}}.
\]
Using \eqref{eq:semigroupFirstEignenvalue} and Fact \ref{fac:backBoneFacts}
one checks that \textbackslash{}
\[
\norm{z^{i,t}}{}{}\cleq e^{-(n-1)(\alpha-\mu)t}\sinh((n-1)\alpha t)e^{-(n-1)\mu t}\norm{z_{t}(i)}{}{}\cleq\norm{z_{t}(i)}{}{}.
\]
Again by Fact \ref{fac:backBoneFacts} we have $\ev{}\rbr{(Z^{i,t})^{2}|z_{t}}\leq(1+\norm{z_{t}}{}{3n})$.
Therefore, under assumption $\alpha>2\mu$ we get 
\[
\ev{}I_{2}(t)^{2}\cleq e^{-2(\alpha-\mu)t}e^{\alpha t}=e^{(-\alpha+2\mu)t}\conv0.
\]
Finally, $z^{i,t}-z_{t}(i)=\frac{\beta}{\alpha}e^{-2(n-1)\alpha t}z_{t}(i)$
which implies that $I_{3}(t)\conv0$. These together yield that \eqref{eq:martingaleEquiv2}.
\eqref{eq:martingaleEquiv1} can be proven using a very similar, but
simpler, reasoning hence is skipped. 
\end{proof}
In the following proof we will need 
\begin{lem}
\label{lem:LaplacePositive}Let $X$ be an a.s. positive random variable.
If there exists an analytic function $w:(-\infty,a)\mapsto\R_{+}$,
$a>0$ such that 
\[
w(\theta)=\ev{e^{\theta X}},\quad\text{for }\theta\leq0,
\]
then \textup{$w(\theta)=\ev{e^{\theta X}}$ holds on $(-\infty,a_{0})$
for some $a_{0}>0$. }\end{lem}
\begin{proof}
There exists a sequence $\cbr{a_{n}}_{n\geq0}$ such that $w(\theta)=\sum_{n\geq0}\frac{a_{n}}{n!}\theta^{n}$
and the series is absolutely convergent in some interval $[-\epsilon,\epsilon]$.
For any $\lambda\in(-\epsilon,0)$ we have (we use the positivity
here): 
\[
w^{(n)}(\theta)=\ev{}X^{n}e^{\theta X}.
\]
Passing to the limit (which is valid by the monotone Lebesgue theorem)
we get 
\[
a_{n}=w^{(n)}(0)=\ev{}X^{n}.
\]
Therefore all the moments exists. Let us now take some $\theta\in(0,\epsilon)$
obviously we have
\[
w(\theta)=\lim_{k\rightarrow+\infty}\sum_{n=0}^{k}\frac{a_{n}}{n!}\theta^{n}=\lim_{k\rightarrow+\infty}\sum_{n=0}^{k}\frac{\ev{X^{n}}}{n!}\theta^{n}=\sum_{n=0}^{\infty}\ev{}\rbr{\frac{X^{n}}{n!}\theta^{n}}=\ev{}\sum_{n=0}^{\infty}\frac{X^{n}}{n!}\theta^{n}=\ev{}e^{\theta X},
\]
where we used the Fubini theorem. Extending the equality to the whole
negative axis follows by a standard argument. 
\end{proof}

\begin{fact}
\label{fact:exponentailMomentsofTotalMass}Let $T>0$, there exists
$\theta>0$ such that 
\[
\ev{e^{\theta|X_{t}|}}<+\infty,
\]
for any $t<T$.\end{fact}
\begin{proof}
We denote $w(\theta,t):=\ev{e^{\theta|X_{t}|}}$, which so far is
well-defined, for $\theta<0$. Using \eqref{eq:totalMassSolution}
it is easy to check that for any $T>0$ there exists $\epsilon>0$
the functions above are analytic on $(-\epsilon,\epsilon)$ (one has
to ensure that the denominators are bounded away from $0$). Now the
conclusion holds by Lemma \ref{lem:LaplacePositive}.
\end{proof}

\section*{Acknowledgments}

The parts of this paper were written while the author enjoyed a kind
hospitality of the Probability Laboratory at Bath. The author wishes
to thank Simon Harris and Andreas Kyprianou for stimulating discussions. 

\bibliographystyle{plain}
\bibliography{/Users/piotrmilos/Dropbox/Sync/Library/branching}

\end{document}